%% file: confinv_312lattice.tex
\documentclass[11pt]{article}
\usepackage{authblk}
\usepackage[toc,page]{appendix}
\usepackage[top=2cm, bottom=2cm, left=2cm, right=2cm]{geometry}

\usepackage{color}
\usepackage{helvet}         
\usepackage{courier}        
\usepackage{type1cm}        

\usepackage{framed} 
\usepackage{tikz}
\usepackage{makeidx}         
\usepackage{graphicx}        
\usepackage{multicol}        
\usepackage[bottom]{footmisc}

\usepackage{amsmath}
\usepackage{amssymb}
\usepackage{bbold}
\usepackage{amsthm}
\usepackage{subcaption}
\usepackage{sidecap}
\usepackage{floatrow}
\usepackage{pdflscape}
\usepackage{comment}
\usepackage[font=small]{caption}
\usepackage{enumitem}
\usepackage{esint}

\usepackage[hidelinks]{hyperref}

\usepackage{scalerel}
\usepackage{shuffle}
\usepackage{mathrsfs}

\sidecaptionvpos{figure}{c} 
\usepackage{sidecap}

\newtheorem{theorem}{Theorem}

\newtheorem{lemma}[theorem]{Lemma}

\newtheorem{proposition}[theorem]{Proposition}

\theoremstyle{definition}
\newtheorem{definition}[theorem]{Definition}
\renewcommand{\arraystretch}{1.5}

\numberwithin{theorem}{section}
\numberwithin{figure}{section}
\numberwithin{table}{section}
\numberwithin{equation}{section}

\begin{document}

\title{Ising model and percolation: from hexagonal lattice to 3-12 lattice}
\bigskip{}
\author[1]{Junyu Mou\thanks{moujy2004@gmail.com.}}
\author[1]{Hao Wu\thanks{hao.wu.proba@gmail.com.}}
\affil[1]{Tsinghua University, China}
\date{}

%
%

\input{tex/macros}
\maketitle
\begin{center}
\begin{minipage}{0.95\textwidth}
\abstract{
In this survey, we extend the conformal invariance of the Ising model and of the percolation from the hexagonal lattice to the 3-12 lattice. 
}
\medbreak
\end{minipage}
\end{center}

\section{Introduction}
\input{tex/intro}
\section{Ising model}
\label{sec::Ising}
\input{tex/Ising}

\section{Percolation}
\label{sec::perco}
\input{tex/perco}
\section{Discussion for other lattices}
\label{sec::otherlattice}
\input{tex/otherlattice}

\appendix
\section{Pure partition functions}
\input{tex/ppf}

\end{document}

%% file: tex/macros.tex
\global\long\def\eps{\epsilon}
\global\long\def\chamber{\mathfrak{X}}
\global\long\def\Y{\mathrm{Y}}

\global\long\def\dist{\mathrm{dist}}
\global\long\def\bs{\boldsymbol}
\global\long\def\SLE{\mathrm{SLE}}
\global\long\def\LP{\mathrm{LP}}
\global\long\def\ee{\mathrm{e}}
\global\long\def\ii{\mathrm{i}}
\global\long\def\312{3\text{-}12}
\global\long\def\Ising{\mathrm{Ising}}
\global\long\def\Perco{\mathrm{perco}}
\global\long\def\PO{\mathrm{H}}
\global\long\def\IO{\mathrm{F}}

\global\long\def\A{\mathbb{A}}
\global\long\def\B{\mathbb{B}}
\global\long\def\C{\mathbb{C}}
\global\long\def\D{\mathbb{D}}
\global\long\def\E{\mathbb{E}}
\global\long\def\F{\mathbb{F}}
\global\long\def\G{\mathbb{G}}
\global\long\def\HH{\mathbb{H}}
\global\long\def\I{\mathbb{I}}
\global\long\def\J{\mathbb{J}}
\global\long\def\K{\mathbb{K}}
\global\long\def\L{\mathbb{L}}
\global\long\def\M{\mathbb{M}}
\global\long\def\N{\mathbb{N}}
\global\long\def\O{\mathbb{O}}
\global\long\def\PP{\mathbb{P}}
\global\long\def\Q{\mathbb{Q}}
\global\long\def\R{\mathbb{R}}
\global\long\def\S{\mathbb{S}}
\global\long\def\T{\mathbb{T}}
\global\long\def\U{\mathbb{U}}
\global\long\def\V{\mathbb{V}}
\global\long\def\W{\mathbb{W}}
\global\long\def\X{\mathbb{X}}
\global\long\def\Z{\mathbb{Z}}
\global\long\def\LA{\mathcal{A}}
\global\long\def\LB{\mathcal{B}}
\global\long\def\LC{\mathcal{C}}
\global\long\def\LD{\mathcal{D}}
\global\long\def\LE{\mathcal{E}}
\global\long\def\LF{\mathcal{F}}
\global\long\def\LG{\mathcal{G}}
\global\long\def\LH{\mathcal{H}}
\global\long\def\LI{\mathcal{I}}
\global\long\def\LJ{\mathcal{J}}
\global\long\def\LK{\mathcal{K}}
\global\long\def\LL{\mathcal{L}}
\global\long\def\LM{\mathcal{M}}
\global\long\def\LN{\mathcal{N}}
\global\long\def\LO{\mathcal{O}}
\global\long\def\LQ{\mathcal{Q}}
\global\long\def\LR{\mathcal{R}}
\global\long\def\LS{\mathcal{S}}
\global\long\def\LT{\mathcal{T}}
\global\long\def\LU{\mathcal{U}}
\global\long\def\LV{\mathcal{V}}
\global\long\def\LW{\mathcal{W}}
\global\long\def\LX{\mathcal{X}}
\global\long\def\LY{\mathcal{Y}}
\global\long\def\LZ{\mathcal{Z}}

%% file: tex/intro.tex
Conformal invariance in critical planar statistical physics models is one of the most fundamental questions in the field. Since 2000, mathematicians have rigorously established conformal invariance for several critical models. In this survey, we focus on the Ising model~\cite{ChelkakSmirnovIsing} and the percolation~\cite{SmirnovPercolationConformalInvariance}. While conformal invariance itself suggests universality—meaning the scaling limit should be independent of the choice of lattice—the existing mathematical proofs rely heavily on the specific lattice structure. In this survey, we extend the known results of conformal invariance from the hexagonal lattice to the 3-12 lattice. 

\paragraph*{Hexagonal lattice and 3-12 lattice.}
The hexagonal lattice is a graph where each vertex has three neighbors, forming regular hexagon as the fundamental unit. 
Fisher transformation acts on each vertex $v$ by replacing it by a triangle. The 3-12 lattice is the Fisher transformation of the hexagonal lattice, as shown in Figure~\ref{fig::Correspondence_6vs312}. 
Fisher transformation was introduced by M. E. Fisher~\cite{Fisher1966} in the study of the relation between Ising and dimer models.
Denote by $\mathbb{L}_6$ the hexagonal lattice and by $\mathbb{L}_{\312}$ the 3-12 lattice.
\begin{SCfigure}[1][!ht]
        \includegraphics[width=0.4\textwidth]{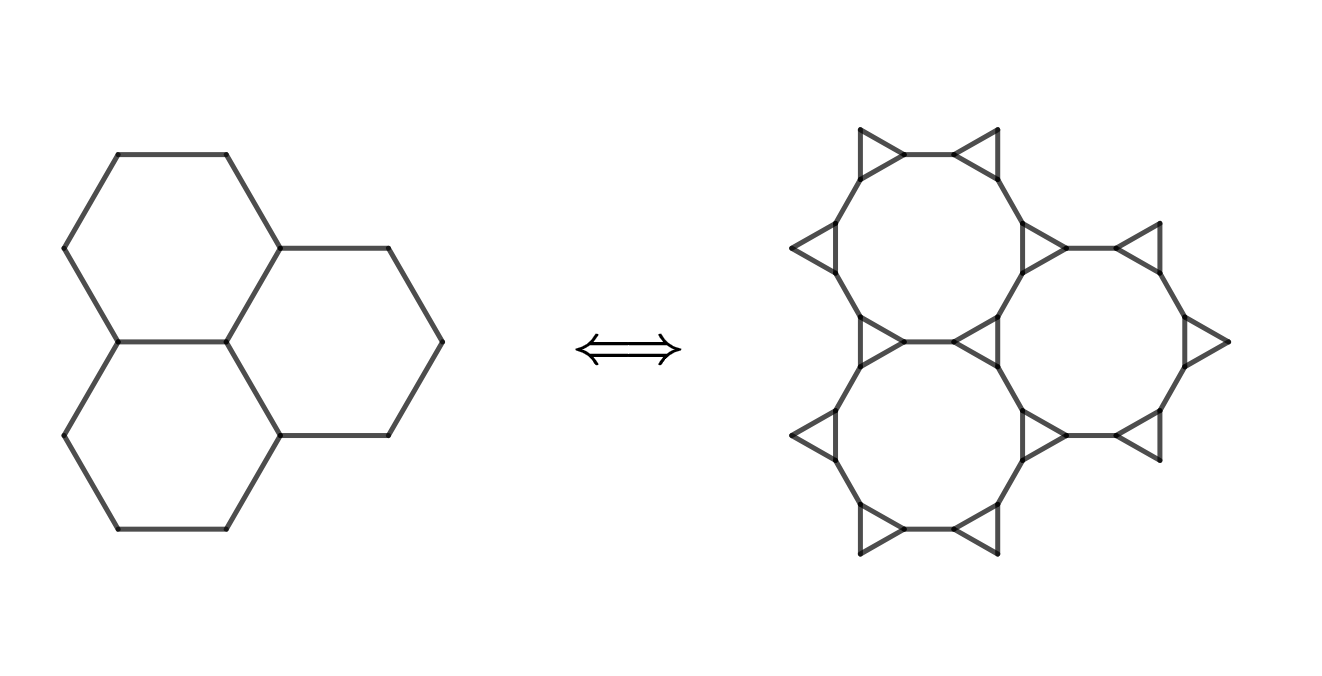}
        \caption{The correspondence between hexagonal lattice (left) and 3-12 lattice (right).}
        \label{fig::Correspondence_6vs312}
\end{SCfigure}

We will consider Ising model and percolation on the hexagonal lattice and on the 3-12 lattice. 
The critical values of these models depend on the lattice, see Table~\ref{table::criticalvalue}; whereas, the critical behavior is the same for both lattices. In these two models, our proof relies on the mapping between vertex configurations observed by M. T. Batchelor in~\cite{Batchelor1998}.
Such mapping was used to derive the connective constant of self-avoiding walk on 3-12 lattice~\cite{JensenGuttmann, Batchelor1998, grimmett2012self} from the connective constant on the hexagonal lattice~\cite{NienhuisOnin2D, DCSmirnovConnectiveConstant}. It turns out that such mapping could also be used to extend the conformal invariance of the Ising model and of the percolation from hexagonal lattice to 3-12 lattice.

\begin{itemize}
\item Ising model. The conformal invariance for the critical Ising model on isoradial graph was proved in~\cite{ChelkakSmirnovIsing}. We extend this conclusion for 3-12 lattice (which is not isoradial). More precisely, we will show that the observable for Ising model introduced in~\cite{ChelkakSmirnovIsing} converge on 3-12 lattice in Theorem~\ref{thm::Ising_observable} and the connection probabilities for the Ising model in polygons converge on 3-12 lattice in Proposition~\ref{prop::Ising_crossproba}. 
\item Percolation. The conformal invariance for the critical site percolation on the hexagonal lattice was proved in~\cite{SmirnovPercolationConformalInvariance}. We extend this conclusion for 3-12 lattice. More precisely, we will show that the observable for site percolation introduced in~\cite{SmirnovPercolationConformalInvariance} converge on 3-12 lattice in Theorem~\ref{thm::Perco_observable} and the connection probabilities for the site percolation in polygons converge on 3-12 lattice in Proposition~\ref{prop::Perco_crossproba}. 
\end{itemize}

\begin{table}[h]
    {\renewcommand{\arraystretch}{1.5} 
    \begin{tabular}{|c|c|c|}
    \hline
    & Hexagonal lattice& 3-12 lattice\\ 
    \hline     
    Ising model & $\beta_6=\frac{1}{4}\log 3$ & $\beta_{\312}=\frac{1}{2}\log \frac{1+\sqrt{4\sqrt{3}-3}}{2}$\\
    \hline
     Site percolation& $p=1/2$ & $p=1/2$\\
    \hline
    \end{tabular}}
    \caption{\label{table::criticalvalue}Critical values on the hexagonal lattice and on the 3-12 lattice. }
\end{table}

\subsection{Ising model} 
\label{subsec::Intro_Ising}

\paragraph*{Ising model.}
Given a finite graph $G=(V(G),E(G))$, a spin assignment on $G$ is a $\sigma=(\sigma^v)_{v\in V(G)}$ with $\sigma^v\in \{-1,1\}$. 
Spin Ising model on $G$ is a random spin assignment such that 
\begin{align}\label{def::Ising}
\PP_G^\beta [\sigma]\propto \exp \left(-\beta H_G(\sigma)\right), \qquad\text{where }H_G(\sigma):=-\sum_{(i,j)\in E(G)}\sigma^i \sigma^j,
\end{align}
where $\beta>0$ is called the inverse temperature and $H_G(\sigma)$ is called the Hamiltonian. 

Given a planar graph $G=(V,E)$, the dual graph $G^*=(V^*,E^*)$ of $G$ is a graph that has a vertex for each face of $G$ and has an edge for each pair of faces in $G$ that are separated by an edge of $G$.
The dual lattice $\mathbb{L}_{\312}^*$ of the 3-12 lattice consists of isosceles triangles with three inner angles $\frac{\pi}{6},\frac{\pi}{6},\frac{2\pi}{3}$ as shown in Figure~\ref{fig::discrete_312}. The vertices with degree $12$ form an equilateral triangular lattice, and vertices with degree $3$ are centers of the equilateral triangles.

\paragraph*{Discrete domains.}
Fix $m\ge 1$, we call $(\Omega; y_1, \ldots, y_m)$ an $m$-polygon if $\Omega$ is a simply connected domain $\Omega\subset \mathbb{C}$, its boundary $\partial\Omega$ is locally connected and the boundary points $y_1,\cdots,y_m$ appearing in counterclockwise order along $\partial\Omega$. We denote by $(y_1y_2)$ the boundary arc from $y_1$ to $y_2$ in counterclockwise orientation.

Denote by $\mathbb{L}_{\312}^\delta$ the $\delta$-scaled 3-12 lattice. A discrete domain $\Omega_{\312}^\delta$ is formed as follows. The boundary is a sequence $\partial \Omega_{\312}^\delta=(u_0,u_1,\ldots,u_{M-1},u_M=u_0)$ consisting of midpoints of edges in $\mathbb{L}_{\312}^\delta$ separating two dodecagons, such that $u_i$ and $u_{i+1}$ belong to a same dodecagon and $\partial\Omega_{\312}^\delta$ forms a loop. The discrete lattice $\Omega_{\312}^\delta$ is the induced subgraph of $\mathbb{L}_{\312}^\delta$ with vertices surrounded by $\partial\Omega_{\312}^\delta$.  The dual lattice $\Omega_{\312}^{\delta,*}$ is the induced subgraph of $\mathbb{L}_{\312}^{\delta,*}$ with vertices at centers of triangles and dodecagons that have at least one vertex in $V(\Omega_{\312}^\delta)$. See in Figure~\ref{fig::discrete_312}.

A discrete polygon $(\Omega_{\312}^{\delta}; y_1^{\delta}, \ldots, y_m^{\delta})$ is a discrete domain $\Omega_{\312}^\delta$ with $m$ boundary points $y_{1}^\delta,y_{2}^\delta,\ldots,y_m^\delta$ in $\partial \Omega_{\312}^\delta$ appearing in counterclockwise order along $\partial \Omega_{\312}^\delta$.
We say discrete polygons $(\Omega_{\312}^{\delta}; y_1^{\delta}, \ldots, y_m^{\delta})$ converge to $(\Omega; y_1, \ldots, y_m)$ as $\delta\to 0$ in the Carath\'{e}odory sense if there exist conformal maps $\phi^{\delta}$ from the unit disc $\U$ onto $\Omega_{\312}^{\delta}$ and a conformal map $\phi$ from $\U$ onto $\Omega$ such that $\phi^{\delta}\to\phi$ uniformly on compact subset of $\U$ as $\delta\to 0$ and $(\phi^{\delta})^{-1}(y_j^{\delta})\to \phi^{-1}(y_j)$ for all $1\le j\le m$.
We say $(\Omega_{\312}^\delta; y_1^\delta,\ldots,y_m^\delta)$ converges to $(\Omega;y_1,\ldots,y_m)$ in the close-Carathéodory sense if it converges in the Carathéodory sense and for each point $z\in \Omega$ and $r>0$ small enough, let $S_r$ be the arc of $\partial B(y_j,r)\cap \Omega$ disconnecting $y_j$ from $z$ and from all other arcs of this set, then when $\delta$ is small enough, $y_j^\delta$ is connected to the midpoint of $S_r$ inside $\Omega_{\312}^{\delta}\cap B(y_j,r)$.

\begin{SCfigure}[1][!ht]
    \centering
    \includegraphics[width=0.4\textwidth]{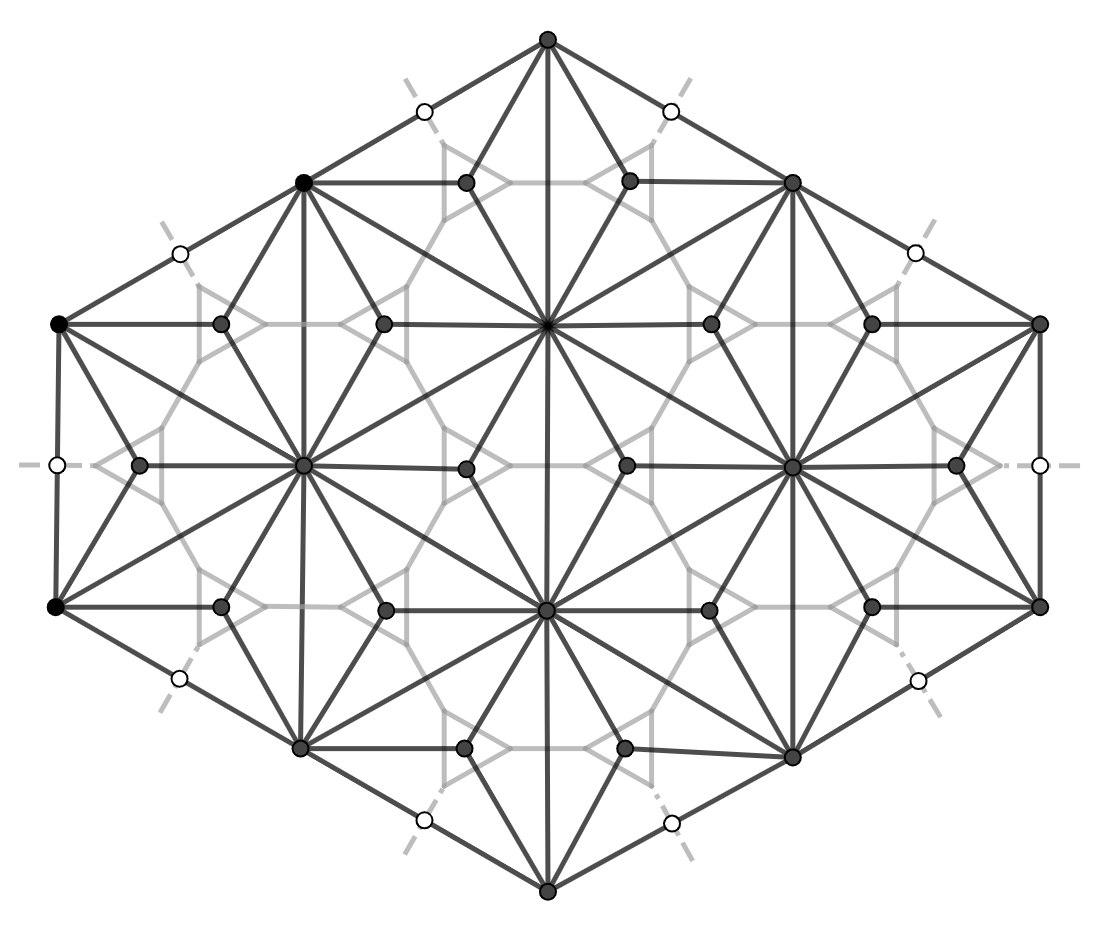}
    \caption{A discrete domain $\Omega_{\312}^\delta$. Grey edges are edges of the primal lattice $\Omega_{\312}^\delta$, black vertices and edges form the dual lattice $\Omega_{\312}^{\delta,*}$, and white points are points in $\partial \Omega_{\312}^\delta$.}
    \label{fig::discrete_312}
\end{SCfigure}

\paragraph*{Chelkak-Smirnov's observable.}
Fix a $2$-polygon $(\Omega; A, B)$. 
Assume that $\Omega$ is flat at $B$: there exists $\eps>0$ such that $[-\eps, \eps]\times (0,\eps]=(-B+\Omega)\cap [-\eps, \eps]^2$. 
Suppose $(\Omega_{\312}^\delta;A^\delta,B^\delta)$ is a family of discrete $2$-polygons such that $(\Omega_{\312}^\delta;A^\delta,B^\delta)$ converges to $(\Omega; A, B)$ in the Carathéodory sense as $\delta\to 0$. We also assume that $\Omega_{\312}^\delta$ is flat near $B^\delta$.

We say a point $z$ is a mid-edge in $\Omega_{\312}^\delta$, if $z$ is the midpoint of some edge in $\Omega_{\312}^\delta$ or $z$ is a point in $\partial \Omega_{\312}^\delta$. Given two mid-edges $z_1,z_2$ in $\Omega_{\312}^\delta$, denote by $\LG_{\312}^\delta(z_1,z_2)$ the collection of subgraphs of $\Omega_{\312}^\delta\cup \{z_1,z_2\}$ that consists of some loops and a self-avoiding path from $z_1$ to $z_2$, such that the loops and the self-avoiding path are disjoint. Given $w\in \LG_{\312}^\delta(z_1,z_2)$, denote by $\gamma(w)$ the self-avoiding path from $z_1$ to $z_2$, and denote by $W_{\gamma(w)}(z_1,z_2)$ the winding number of $\gamma(w)$, which is the total signed rotation in radians of the oriented path.
We consider Ising model with inverse temperature $\beta$ on $(\Omega_{\312}^{\delta,*}; A^{\delta}, B^{\delta})$ with Dobrushin boundary conditions: the spins on $(A^{\delta}B^{\delta})$ are $+1$ and the spins on $(B^{\delta}A^{\delta})$ are $-1$.
We define the fermionic observable to be
\begin{align}\label{eqn::Ising_observable}
    \IO_{\312}^{\delta}(x;z^\diamond)=\dfrac{\sum\limits_{w\in \LG_{\312}^\delta(A^\delta,z^\diamond)}x^{|w|}\exp\left(-\frac{\ii}{2}W_{\gamma(w)}(A^\delta,z^\diamond)\right)}{\sum\limits_{w\in \LG_{\312}^\delta(A^\delta,B^\delta)}x^{|w|}\exp\left(-\frac{\ii}{2}W_{\gamma(w)}(A^\delta,B^\delta)\right)}, \qquad \text{where }x=\ee^{-2\beta}, 
\end{align}
and $z^\diamond$ is a midpoint of an edge in $\Omega_{\312}^\delta$ separating two dodecagons and $|w|$ denotes the number of vertices of $\Omega_{\312}^\delta$ in $w$. 

\begin{theorem}\label{thm::Ising_observable}
Fix a $2$-polygon $(\Omega; A, B)$ and assume that $\Omega$ is flat at $B$. 
Let $\phi$ be a conformal map from $\Omega$ onto the upper-half plane $\HH$ such that $\phi(A)=\infty$ and $\phi(B)=0$. Define $\IO(z)=\sqrt{\phi'(z)/\phi'(B)}$. 
Suppose $(\Omega_{\312}^\delta;A^\delta,B^\delta)$ is a family of discrete $2$-polygons such that $(\Omega_{\312}^\delta;A^\delta,B^\delta)$ converges to $(\Omega; A, B)$ in the Carath\'{e}odory sense as $\delta\to 0$ and also assume that $\Omega_{\312}^\delta$ is flat near $B^\delta$. 
Set
\begin{align}\label{eqn::Ising_criticalbeta_312}
\beta_{\312}=\frac{1}{2}\log \frac{1+\sqrt{4\sqrt{3}-3}}{2}, \qquad x_{\312}=\frac{2}{1+\sqrt{4\sqrt{3}-3}}.
\end{align}
Consider Ising model with inverse temperature $\beta_{\312}$ in $(\Omega_{\312}^\delta;A^\delta,B^\delta)$ with Dobrushin boundary conditions.
Then the observable $\IO_{\312}^{\delta}(x_{\312}; \cdot)$ defined in~\eqref{eqn::Ising_observable} converges to $\IO$ uniformly on any compact subset of $\Omega$ as $\delta\to 0$.
\end{theorem}

\paragraph*{Multiple interfaces of Ising model.} 
Given a discrete $2N$-polygon $(\Omega_{\312}^\delta;y_1^\delta,\ldots,y_{2N}^\delta)$, we consider alternating boundary conditions: 
the boundary vertices of $\partial\Omega_{\312}^{\delta,*}$ on the arc $(y_{2j-1}^{\delta} y_{2j}^{\delta})$ have spin $+1$ for $j\in \{1, \ldots, N\}$; 
the boundary vertices of $\partial\Omega_{\312}^{\delta,*}$ on the arc $(y_{2j}^{\delta} y_{2j+1}^{\delta})$ have spin $-1$ for $j\in \{0, \ldots, N-1\}$. 
Here we write $y_0=y_{2N}$. 
Given a spin assignment $\sigma$ with alternating boundary conditions, for any $1\le j\le N$, there exists a unique self-avoiding path $\gamma_{\312}^{\delta,j}$ on $\Omega_{\312}^\delta\cup \{y_1^\delta,\ldots,y_{2N}^\delta\}$ starting from $y_{2j-1}^\delta$ and ending at some point $y_{2k_j}^\delta$ in $\{y_{2}^{\delta}, y_4^{\delta}, \ldots, y_{2N}^{\delta}\}$, such that for each oriented edge $e$ of $\gamma_{\312}^{\delta,j}$, the center of the face on the right-hand side of $e$ has spin $+1$ and the center of the face on the left-hand side of $e$ has spin $-1$. Note that $\gamma_{\312}^{\delta,1},\ldots,\gamma_{\312}^{\delta,N}$ are disjoint, see Figure~\ref{fig::interface_312}. 
We call $\LA^{\delta}_{\312}=\{\{1,2k_1\},\{3,2k_2\},\ldots,\{2N-1,2k_N\}\}$ a link pattern, which gives a pair partition of $\{1,\ldots,2N\}$. We denote by $\mathrm{LP}_N$ the set of all possible link patterns. 

\begin{SCfigure}[1][!ht]
    \includegraphics[width=0.4\textwidth]{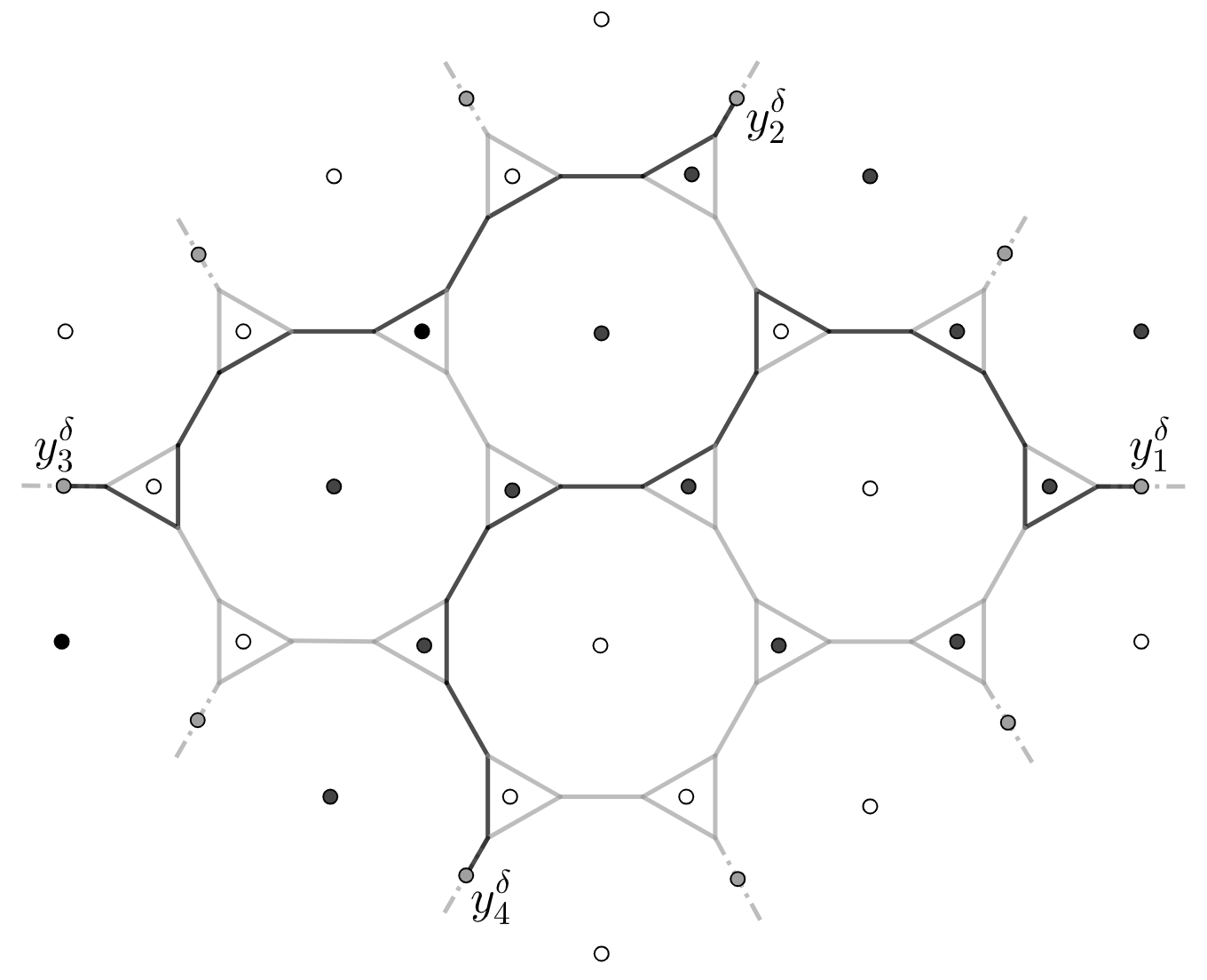}
    \caption{A spin assignment with alternating boundary condition for $N=2$. Vertices with spin $+1$ are colored black and vertices with spin $-1$ are colored white. Black edges form the multiple interface for the spin assignment, and the corresponding link pattern is $\LA^{\delta}_{\312}=\{\{1,4\},\{2,3\}\}$.}
    \label{fig::interface_312}
\end{SCfigure}

\begin{proposition}\label{prop::Ising_crossproba}
Fix $N\ge 1$ and $2N$-polygon $(\Omega;y_1,\ldots,y_{2N})$. 
Suppose $(\Omega_{\312}^\delta;y_1^\delta,\ldots, y_{2N}^\delta)$ is a family of discrete polygons such that $(\Omega_{\312}^\delta;y_1^\delta,\ldots,y_{2N}^\delta)$ converges to $(\Omega;y_1,\ldots,y_{2N})$ in the close-Carathéodory sense as $\delta\to 0$. 
Let $\LA^\delta_{\312}$ be the link pattern for Ising model with inverse temperature $\beta_{\312}$ (defined in~\eqref{eqn::Ising_criticalbeta_312}) on $(\Omega_{\312}^{\delta, *}; y_1^\delta,\ldots, y_{2N}^\delta)$ with alternating boundary conditions. Then the scaling limits for probabilities for the link pattern are conformally invariant:
\begin{align}\label{eqn::Ising_crossproba}
\lim_{\delta\to 0}\PP_{\Ising}\left[\LA_{\312}^{\delta}=\alpha\right]=\frac{\LZ_{\alpha}^{(3)}(\Omega; y_1, \ldots, y_{2N})}{\LZ_{\mathrm{Ising}}(\Omega; y_1, \ldots, y_{2N})}, \qquad\text{for all }\alpha\in\LP_N, 
\end{align}
where $\{\LZ_{\alpha}^{(3)}: \alpha\in\LP_N\}$ are the pure partition functions for multiple $\SLE_3$ given in Definition~\ref{def::pure_partition} and $\LZ_{\mathrm{Ising}}=\sum_{\alpha\in\LP_N}\LZ_{\alpha}^{(3)}$. 
\end{proposition}

\subsection{Percolation}
\label{subsec::intro_perco}
\paragraph*{Site percolation.} 
Given a finite graph $G=(V(G),E(G))$, the site percolation with probability $p\in [0,1]$ is a random spin assignment $\sigma$ on $G$ such that $(\sigma^v)_{v\in V(G)}$ are i.i.d. with \[\PP_{\Perco}[\sigma^v=1]=p,\quad \PP_{\Perco}[\sigma^v=-1]=1-p.\]
Equivalently, the site percolation measure $\PP_{\Perco}$ on the space of spin assignments is given by 
\[\PP_{\Perco}[\sigma]=p^{\# \{v\in V(G):\sigma^v=1\}}(1-p)^{\# \{v\in V(G):\sigma^v=-1\}}.\]

\paragraph*{Smirnov's observable.}
Fix a $3$-polygon $(\Omega; A, B, C)$. Suppose $(\Omega_{\312}^{\delta}; A^{\delta}, B^{\delta}, C^{\delta})$ is a family of discrete polygons such that $(\Omega_{\312}^{\delta}; A^{\delta}, B^{\delta}, C^{\delta})$ converges to $(\Omega; A, B, C)$ in the Carathéodory sense as $\delta\to 0$.
We consider site percolation with $p=1/2$. For a vertex $z\in V(\Omega_{\312}^{\delta})$, define $\LE_{\312}^{\delta}(A; z)$ to be the event that there exists a path $\gamma$ on $\Omega_{\312}^{\delta,*}$ connecting boundary arcs $(A^\delta B^\delta)$ and $(C^\delta A^\delta)$, such that all vertices on $\gamma$ have spin $+1$ and $\gamma$ separates $z$ from the boundary arc $(B^\delta C^\delta)$. 
We denote by $\PO_{\312}^{\delta}(A; z)$ the probability of the event $\LE_{\312}^{\delta}(A; z)$.
We define the events $\LE_{\312}^\delta(B;z),\LE_{\312}^\delta(C;z)$ and the probabilities $\PO_{\312}^\delta(B;z),\PO_{\312}^\delta(C;z)$ in a similar way.
We define the observable
\begin{align}\label{eqn::perco_observable}
\PO_{\312}^{\delta}(z):=\PO_{\312}^{\delta}(A; z)+\tau \PO_{\312}^{\delta}(B; z)+\tau^2 \PO_{\312}^{\delta}(C; z), \qquad \text{where }\tau=\ee^{2\pi\ii/3}. 
\end{align}

\begin{theorem}\label{thm::Perco_observable}
Fix a bounded $3$-polygon $(\Omega; A, B, C)$. Suppose $\PO$ is the unique conformal map from $\Omega$ onto the equilateral triangle $\rhd$ with three vertices $(1, \tau, \tau^2)$ such that $\PO(A)=1, \PO(B)=\tau, \PO(C)=\tau^2$. 
Suppose $(\Omega_{\312}^{\delta}; A^{\delta}, B^{\delta}, C^{\delta})$ is a family of discrete polygons such that $(\Omega_{\312}^{\delta}; A^{\delta}, B^{\delta}, C^{\delta})$ converges to $(\Omega; A, B, C)$ in the Carathéodory sense as $\delta\to 0$. 
Consider site percolation with $p=1/2$ in $\Omega_{\312}^{\delta, *}$. 
Then the observable $\PO_{\312}^{\delta}$ defined in~\eqref{eqn::perco_observable} converges to $\PO$ uniformly as $\delta\to 0$. 
\end{theorem}

\paragraph*{Convergence of crossing probabilities.}
We use the same notion of the discretization of a polygon $(\Omega;y_1,\ldots,y_{2N})$ in Section~\ref{subsec::Intro_Ising}. 
We also use the same construction of the multiple interfaces and the link pattern of a spin assignment on the dual lattice with alternating boundary conditions in Section~\ref{subsec::Intro_Ising}. Analogous conclusion for percolation also holds.

\begin{proposition}\label{prop::Perco_crossproba}
Fix $N\ge 1$ and $2N$-polygon $(\Omega;y_1,\ldots,y_{2N})$. 
Suppose $(\Omega_{\312}^\delta;y_1^\delta,\ldots, y_{2N}^\delta)$ is a family of discrete polygons such that $(\Omega_{\312}^\delta;y_1^\delta,\ldots,y_{2N}^\delta)$ converges to $(\Omega;y_1,\ldots,y_{2N})$ in the close-Carathéodory sense as $\delta\to 0$. 
Let $\LA^\delta_{\312}$ be the link pattern for percolation with $p=1/2$ in $(\Omega_{\312}^{\delta, *}; y_1^\delta,\ldots, y_{2N}^\delta)$ with alternating boundary conditions. Then the scaling limits for probabilities for the link pattern are conformally invariant:
\begin{align}\label{eqn::perco_crossproba}
\lim_{\delta\to 0}\PP_{\Perco}\left[\LA_{\312}^{\delta}=\alpha\right]=\LZ_{\alpha}^{(6)}(\Omega; y_1, \ldots, y_{2N}), \qquad\text{for all }\alpha\in\LP_N, 
\end{align}
where $\{\LZ_{\alpha}^{(6)}: \alpha\in\LP_N\}$ are the pure partition functions for multiple $\SLE_6$ given in Definition~\ref{def::pure_partition}. 
\end{proposition}
When $N=2$, the convergence~\eqref{eqn::perco_crossproba} is Cardy-Smirnov formula~\cite{CardyCriticalPerco,SmirnovPercolationConformalInvariance}. 

\paragraph*{Outline and strategy.}
We prove Theorem~\ref{thm::Ising_observable} and Proposition~\ref{prop::Ising_crossproba} in Section~\ref{sec::Ising}. 
We present a coupling between Ising models on $\L_6$ and $\L_{\312}$ in Section~\ref{subsec::couplingIsing}, and then complete the proof for the Ising model using the coupling and the known results for $\L_6$ from~\cite{ChelkakSmirnovIsing} and~\cite{PeltolaWuCrossingProbaIsing}.  
We prove Theorem~\ref{thm::Perco_observable} and Proposition~\ref{prop::Perco_crossproba} in Section~\ref{sec::perco}. 
We control the difference between the observable on the two lattices and complete the proof for the percolation using the known results for $\L_6$ from~\cite{SmirnovPercolationConformalInvariance, beffara2007cardy}. 
We also discuss the conformal invariance on the martini lattice in Section~\ref{sec::otherlattice}. 

\paragraph*{Acknowledgment.}
We thank M. T. Batchelor for telling us~\cite{Batchelor1998}. 
We thank G. R. Grimmett for telling us~\cite{grimmett2012self}.
H.W. is supported by New Cornerstone Investigator Program 100001127. H.W. is partly affiliated at Yanqi Lake Beijing Institute of Mathematical Sciences and Applications, Beijing, China.

%% file: tex/Ising.tex

\subsection{Coupling between Ising models}
\label{subsec::couplingIsing}

\paragraph*{Mapping on vertex set.}
Recall that $\mathbb{L}_{\312}=(V(\mathbb{L}_{\312}), E(\mathbb{L}_{\312}))$ denotes the 3-12 lattice. Define an equivalent relation $\sim$ on the vertices $V(\mathbb{L}_{\312})$: $u\sim v$ if $u$ and $v$ belong to the same triangular face. Take the quotient space $\mathbb{L}_{\312}/\sim$, then each equivalence class (3 vertices from one triangle) becomes a single vertex; the edges of $\mathbb{L}_{\312}$ that originally connect vertices within the same triangle are discarded; the edges connecting vertices from different triangles become edges in the quotient, forming a hexagonal lattice. This gives a three-to-one mapping $\mathscr{F}:V(\mathbb{L}_{\312})\to V(\mathbb{L}_6)$ between vertex sets. 

\paragraph*{Discrete polygons.}
Given a discrete polygon $(\Omega_{\312}^\delta; y_1^\delta,\ldots,y_{2N}^\delta)$ on the 3-12 lattice, we construct the corresponding discrete polygon $(\Omega_6^\delta,y_1^\delta,\ldots,y_{2N}^\delta)$ on the hexagonal lattice as follows. 
Let $\Omega_6^\delta$ be the induced subgraph of $\mathbb{L}_6^\delta$ with vertex set $V(\Omega_6^\delta)=\mathscr{F}(V(\Omega_{\312}^\delta))$ and let $\Omega_6^{\delta,*}$ be the induced subgraph of $\Omega_{\312}^{\delta,*}$ that consists of all vertices of equilateral triangles in $\Omega_{\312}^{\delta,*}$. Boundary points $\partial \Omega_6^\delta=\partial \Omega_{\312}^\delta$ are midpoints of edges in $\mathbb{L}_6^*$ that have exactly one endpoint in $V(\Omega_6^\delta)$. In particular boundary points $x_1^\delta,\ldots,x_{2N}^\delta$ for $\Omega_{\312}^{\delta}$ are boundary points for $\Omega_6^\delta$. See Figure~\ref{fig::discrete_6} for a discrete polygon on the hexagonal lattice.

\medbreak
In this section, we will give a coupling between the Ising model on $\Omega_{\312}^{\delta, *}$ and the Ising model on $\Omega_6^{\delta, *}$.
Given a spin assignment $\sigma_{\312}$ on $\Omega_{\312}^{\delta,*}$, we denote the restriction of $\sigma_{\312}$ on $V(\Omega_6^{\delta,*})$ by 
\begin{align}\label{eqn::sigma312_sigma6}
\mathscr{G}(\sigma_{\312}):=\left(\sigma_{\312}^v: v\in V(\Omega_6^{\delta, *})\right). 
\end{align}
This is a spin assignment on $\Omega_6^{\delta,*}$. Note that $\mathscr{G}(\sigma_{\312})$ and $\sigma_{\312}$ have the same boundary conditions because $\partial \Omega_{\312}^{\delta,*}=\partial \Omega_6^{\delta,*}$.

\begin{figure}[h]
    \includegraphics[height=0.3\textwidth]{figures/discrete_312.png}
    \includegraphics[height=0.3\textwidth]{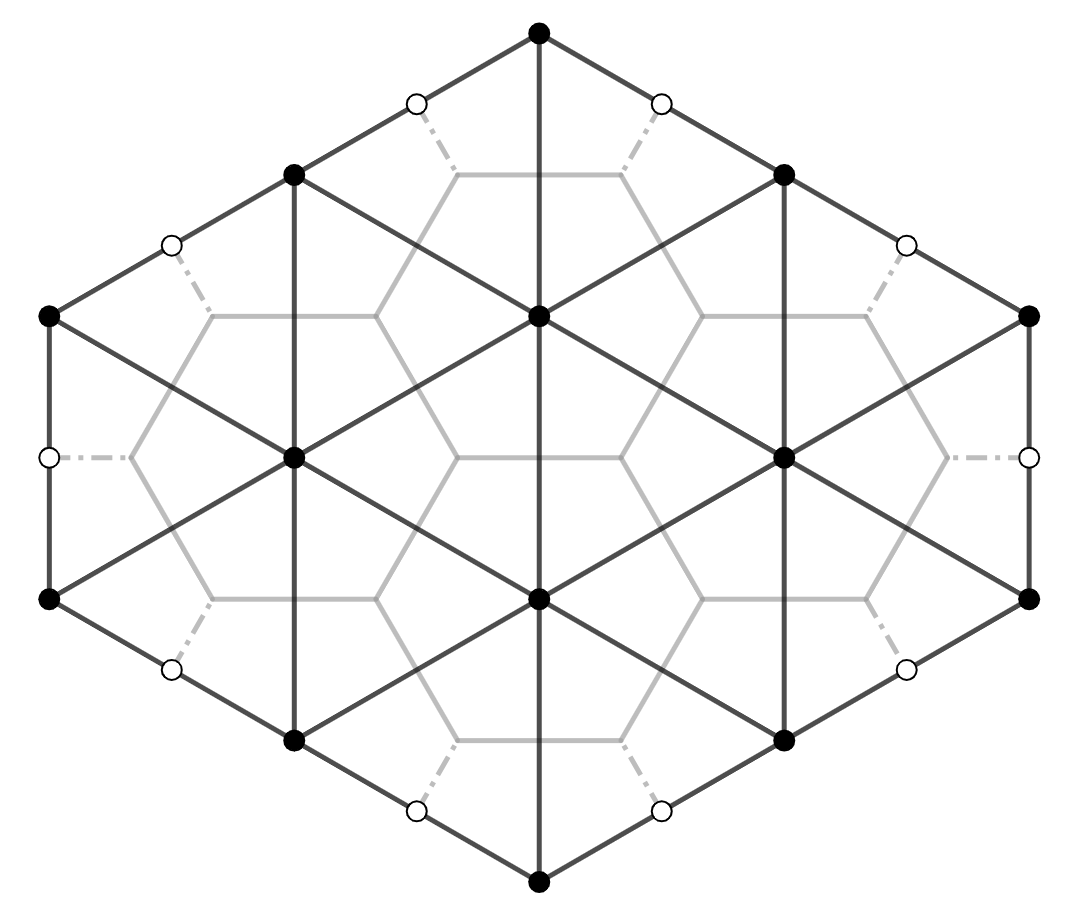}
    \caption{
In the left panel, this is the polyon $\Omega_{\312}^\delta$ in Figure~\ref{fig::discrete_312}. 
In the right panel, this is the corresponding discrete domain $\Omega_6^\delta$ on the hexagonal lattice. Grey edges are edges of the primal lattice $\Omega_6^\delta$, black vertices and edges form $\Omega_6^{\delta,*}$, and white points are points in $\partial \Omega_6^\delta$.} 
    \label{fig::discrete_6}
\end{figure}

\begin{lemma}\label{lem::couplingIsing}
Suppose $\Sigma_{\312}$ has the law of Ising model with inverse temperature $\beta$ on the discrete polygon $(\Omega_{\312}^{\delta, *}; x_1^{\delta}, \ldots, x_{2N}^{\delta})$ with alternating boundary conditions. Then $\mathscr{G}(\Sigma_{\312})$ has the law of Ising model with inverse temperature  
\begin{align}\label{eqn::beta_xi}
\xi=\frac{1}{2}\log (\ee^{4\beta}-\ee^{2\beta}+1)
\end{align}
on the discrete polygon $(\Omega_6^{\delta, *}; y_1^{\delta}, \ldots, y_{2N}^{\delta})$ with alternating boundary conditions.
\end{lemma}
\begin{proof}
For a spin configuration $\sigma_{\312}$ on $(\Omega_{\312}^{\delta, *}; y_1^{\delta}, \ldots, y_{2N}^{\delta})$ with alternating boundary conditions, denote by $H_{\312}^{\delta}(\sigma_{\312})=-\sum_{(i,j)\in E(\Omega_{\312}^{\delta, *})}\sigma_{\312}^i\sigma_{\312}^{j}$ the Hamiltonian of $\sigma_{\312}$ given in~\eqref{def::Ising}. Then
\begin{align}\label{eqn::Ising312_proba}
\PP\left[\Sigma_{\312}=\sigma_{\312}\right]=\frac{1}{Z^{\delta}_{\312; \beta}}\exp\left(-\beta H_{\312}^{\delta}(\sigma_{\312})\right), \qquad \text{where }Z^{\delta}_{\312; \beta}=\sum_\tau \exp\left(-\beta H_{\312}^{\delta}(\tau)\right),
\end{align}
and the summation is over all spin configurations $\tau$ on $(\Omega_{\312}^{\delta, *}; y_1^{\delta}, \ldots, y_{2N}^{\delta})$ with alternating boundary conditions.
For a spin configuration $\sigma_{6}$ on $(\Omega_{6}^{\delta, *}; y_1^{\delta}, \ldots, y_{2N}^{\delta})$ with alternating boundary conditions, denote by $H_{6}^{\delta}(\sigma_{6})=-\sum_{(i,j)\in E(\Omega_{6}^{\delta, *})}\sigma_{6}^i\sigma_{6}^{j}$ the Hamiltonian of $\sigma_{6}$ given in~\eqref{def::Ising}. Our goal is to show that 
\begin{align}\label{eqn::Ising312_induced6}
\begin{split}
&\PP[\mathscr{G}(\Sigma_{\312})=\sigma_6]=\frac{1}{Z_{6;\xi}^{\delta}}\exp\left(-\xi H_6^{\delta}(\sigma_6)\right), \qquad\text{where }
Z^{\delta}_{6; \xi}=\sum_\tau \exp\left(-\xi H_{6}^{\delta}(\tau)\right),
\end{split}
\end{align}
and the summation is over all spin configurations $\tau$ on $(\Omega_{6}^{\delta, *}; y_1^{\delta}, \ldots, y_{2N}^{\delta})$ with alternating boundary conditions. From~\eqref{eqn::Ising312_proba}, we have 
\begin{align*}
\PP[\mathscr{G}(\Sigma_{\312})=\sigma_6]=\frac{1}{Z^{\delta}_{\312; \beta}}\sum_{\substack{\sigma_{\312}: \\\mathscr{G}(\sigma_{\312})=\sigma_6}}\exp\left(-\beta H_{\312}^{\delta}(\sigma_{\312})\right), 
\end{align*}
where the summation is over all spin assignments $\sigma_{\312}$ on $\Omega_{\312}^{\delta, *}$ such that $\mathscr{G}(\sigma_{\312})=\sigma_6$. Comparing with~\eqref{eqn::Ising312_induced6}, it suffices to show 
\begin{align}\label{eqn::Ising312_induced6_aux1}
\sum_{\substack{\sigma_{\312}: \\\mathscr{G}(\sigma_{\312})=\sigma_6}}\exp\left(-\beta H_{\312}^{\delta}(\sigma_{\312})\right)=\frac{Z^{\delta}_{\312;\beta}}{Z^{\delta}_{6;\xi}}\exp\left(-\xi H_6^{\delta}(\sigma_6)\right). 
\end{align}

\begin{SCfigure}[1.5][!ht]
    \includegraphics[width=0.3\textwidth]{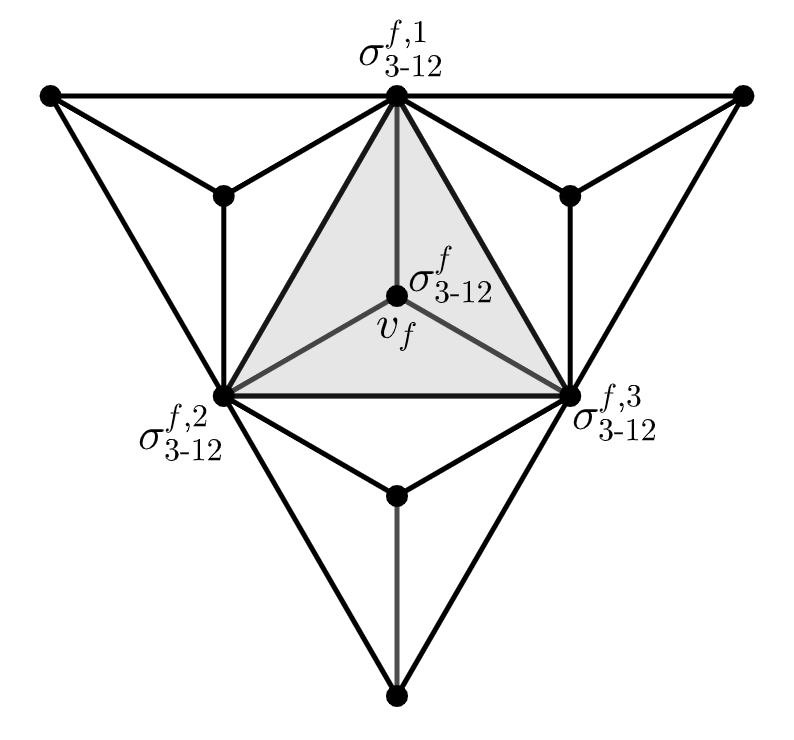}
    \caption{The triangle colored grey is a face $f$ in $\Omega_6^{\delta,*}$ with the center $v_f$. Denote by $\sigma_{\312}^f$ the spin on $v_f$, and by $\sigma_{\312}^{f,1},\sigma_{\312}^{f,2},\sigma_{\312}^{f,3}$ the spins on three neighboring vertices of $v_f$.}
    \label{fig::local_spin}
\end{SCfigure}

Denote by $F(\Omega_6^{\delta,*})$ the set of faces in $\Omega_6^{\delta,*}$. For a face $f\in F(\Omega_6^{\delta,*})$, let $v_f$ be the center of $f$. 
In fact, $v_f$ is a vertex of $\Omega_{\312}^{\delta,*}$ with degree three. For a spin configuration $\sigma_{\312}$ in $\Omega_{\312}^{\delta, *}$, 
we denote its spin at $v_f$ by $\sigma_{\312}^f$ and denote its spins at the three neighboring vertices by $\sigma_{\312}^{f,1}, \sigma_{\312}^{f,2}, \sigma_{\312}^{f,3}$, see Figure~\ref{fig::local_spin}. We define its local energy for $f$ by 
\[H_{\312; f}(\sigma_{\312}):=-(\sigma_{\312}^{f,1}+\sigma_{\312}^{f,2}+\sigma_{\312}^{f,3})\sigma_{\312}^f -\frac{1}{2}(\sigma_{\312}^{f,1}\sigma_{\312}^{f,2}+\sigma_{\312}^{f,1}\sigma_{\312}^{f,3}+\sigma_{\312}^{f,2}\sigma_{\312}^{f,3}).\]
One may check that 
\begin{align}\label{eqn::Ising_aux2}
    H_{\312}^\delta(\sigma_{\312})=\sum_{f\in F(\Omega_6^{\delta,*})}H_{\312;f}(\sigma_{\312})-\underbrace{\frac{1}{2}\sum_{(u,v)\in \partial \Omega_{\312}^{\delta,*}}\sigma_{\312}^u\sigma_{\312}^v}_{H_{\partial}:=},
\end{align}
where the second sum is taken for all edges on the boundary of $\Omega_{\312}^{\delta,*}$. Since the boundary condition is fixed, the quantity $H_{\partial}$ is deterministic.

Similarly, for the spin assignment $\sigma_6$ on $\Omega_6^{\delta,*}$ with spins $\sigma_6^{f,1},\sigma_6^{f,2},\sigma_6^{f,3}$ on the three vertices of a face $f\in F(\Omega_6^{\delta,*})$, define its local energy for $f$ by
\[H_{6;f}(\sigma_6)=-\frac{1}{2}(\sigma_6^{f,1}\sigma_6^{f,2}+\sigma_6^{f,1}\sigma_6^{f,3}+\sigma_6^{f,2}\sigma_6^{f,3}),\]
then the Hamiltonian for $\sigma_6$ on $\Omega_6^{\delta,*}$ is 
\begin{align}\label{eqn::Ising_aux3}
 H_6^\delta(\sigma_6)=\sum_{f\in F(\Omega_6^{\delta,*})}H_{6;f}(\sigma_6)-H_{\partial}.
 \end{align}

For a fixed $\sigma_6$ and a fixed face $f\in F(\Omega_6^{\delta, *})$, let us consider $H_{\312;f}(\sigma_{\312})$ for $\sigma_{\312}$ such that $\mathscr{G}(\sigma_{\312})=\sigma_6$. As the restriction of $\sigma_{\312}$ to $V(\Omega_6^{\delta, *})$ is the same as $\sigma_6$, we have 
\[\sigma_{\312}^{f, 1}=\sigma_6^{f, 1}, \qquad \sigma_{\312}^{f, 2}=\sigma_6^{f, 2}, \qquad \sigma_{\312}^{f, 3}=\sigma_6^{f, 3}, \qquad \sigma_{\312}^f\in \{\pm 1\}.\]
Combining with~\eqref{eqn::Ising_aux2}, we have
\begin{align}\label{eqn::Ising_aux4}
\text{LHS of~\eqref{eqn::Ising312_induced6_aux1}}
    =&\ee^{\beta H_\partial}\sum_{\substack{\sigma_{\312}: \\\mathscr{G}(\sigma_{\312})=\sigma_6}}\exp\left(-\beta \sum_{f\in F(\Omega_6^{\delta,*})}H_{\312;f}(\sigma_{\312})\right)\notag\\
    =&   \ee^{\beta H_\partial}\prod_{f\in F(\Omega_6^{\delta, *})}\sum_{\substack{\sigma_{\312}^f\in\{\pm 1\}\\ \sigma_{\312}^{f, j}=\sigma_6^{f,j}} }\exp\left(-\beta H_{\312; f}(\sigma_{\312})\right).
\end{align}
There are two cases for $\sigma_6^{f,1}, \sigma_6^{f,2}, \sigma_6^{f,3}$.
\begin{itemize}
\item If $\sigma_6^{f,1}, \sigma_6^{f,2}, \sigma_6^{f,3}$ have the same sign, then 
\begin{align*}
H_{6;f}(\sigma_6)=-\frac{3}{2}, \qquad \sum_{\substack{\sigma_{\312}^f\in\{\pm 1\}\\ \sigma_{\312}^{f, j}=\sigma_6^{f,j}} }\exp\left(-\beta H_{\312; f}(\sigma_{\312})\right)=\ee^{\frac{9}{2}\beta}+\ee^{-\frac{3}{2}\beta}.
\end{align*}
\item If $\sigma_6^{f,1}, \sigma_6^{f,2}, \sigma_6^{f,3}$ don't have the same sign, then
\begin{align*}
H_{6;f}(\sigma_6)=\frac{1}{2}, \qquad \sum_{\substack{\sigma_{\312}^f\in\{\pm 1\}\\ \sigma_{\312}^{f, j}=\sigma_6^{f,j}} }\exp\left(-\beta H_{\312; f}(\sigma_{\312})\right)=\ee^{\frac{1}{2}\beta}+\ee^{-\frac{3}{2}\beta}.
\end{align*}
\end{itemize}
Denote by $F^{(1)}(\sigma_6)$ the set of faces of the first case, and denote by $F^{(2)}(\sigma_6)$ the set of faces of the second case. 
Note that $\#F^{(1)}(\sigma_6)+\#F^{(2)}(\sigma_6)=\# F(\Omega_6^{\delta, *})$ which does not depend on $\sigma_6$. 
Plugging the calculation with the two cases into~\eqref{eqn::Ising_aux4}, we have
\begin{align*}
\text{LHS of~\eqref{eqn::Ising312_induced6_aux1}}=&\ee^{\beta H_{\partial}}\times \left(\ee^{\frac{9}{2}\beta}+\ee^{-\frac{3}{2}\beta}\right)^{\#F^{(1)}(\sigma_6)}\times \left(\ee^{\frac{1}{2}\beta}+\ee^{-\frac{3}{2}\beta}\right)^{\# F^{(2)}(\sigma_6)}=C_{\beta}\left(\frac{\ee^{\frac{9}{2}\beta}+\ee^{-\frac{3}{2}\beta}}{\ee^{\frac{1}{2}\beta}+\ee^{-\frac{3}{2}\beta}}\right)^{\#F^{(1)}(\sigma_6)},
\end{align*}
where $C_{\beta}$ is a constant depending on $\beta$ and the polygon $(\Omega_6^{\delta,*}; y_1^{\delta}, \ldots, y_{2N}^{\delta})$. 
Plugging the calculation with the two cases into~\eqref{eqn::Ising_aux3}, we have
\begin{align*}
\exp(-\xi H_6^{\delta}(\sigma_6))=\ee^{\xi H_{\partial}}\times \left(\ee^{\frac{3}{2}\xi}\right)^{\#F^{(1)}(\sigma_6)}\times \left(\ee^{-\frac{1}{2}\xi}\right)^{\# F^{(2)}(\sigma_6)}=C_{\xi}\left(\ee^{2\xi}\right)^{\#F^{(1)}(\sigma_6)}, 
\end{align*}
where $C_{\xi}$ is a constant depending on $\xi$ and the polygon $(\Omega_6^{\delta,*}; y_1^{\delta}, \ldots, y_{2N}^{\delta})$. When $\xi$ is chosen as in~\eqref{eqn::beta_xi}, we have 
\begin{align*}
\text{LHS of~\eqref{eqn::Ising312_induced6_aux1}}=\frac{C_{\beta}}{C_{\xi}}\exp(-\xi H_6^{\delta}(\sigma_6)). 
\end{align*}
Summing over all possible $\sigma_6$ in both sides, we have 
\[Z_{\312;\beta}^{\delta}=\frac{C_{\beta}}{C_{\xi}}Z_{6;\xi}^{\delta}.\]
Therefore,
\[\text{LHS of~\eqref{eqn::Ising312_induced6_aux1}}=\frac{Z_{\312;\beta}^{\delta}}{Z_{6;\xi}^{\delta}}\exp(-\xi H_6^{\delta}(\sigma_6)),\]
as desired. 
\end{proof}

\subsection{Convergence of crossing probabilities: proof of Proposition~\ref{prop::Ising_crossproba}}
\label{subsec::Ising_crossproba}

\paragraph*{Mapping on paths.}
Suppose $z_1,z_2$ are two mid-edges of $\Omega_{\312}^\delta$ separating two dodecagons and suppose $\gamma$ is a path from $z_1$ to $z_2$ on $\Omega_{\312}^\delta$. Let $z^{(0)},\cdots,z^{(\ell)}$ be mid-edges in $\Omega_{\312}^\delta$ separating two dodecagons that are visited by $\gamma$ consecutively. 
Using such notations, we have $z^{(0)}=z_1$ and $z^{(\ell)}=z_2$.
Note that $z^{(0)},\cdots,z^{(\ell)}$ are mid-edges in $\Omega_6^\delta$ (midpoints for some edge in $\Omega_6^\delta$ or boundary points in $\partial \Omega_6^\delta$). We define $\mathscr{F}(\gamma)$ to be the path on $\Omega_{6}^\delta\cup \{z_1,z_2\}$, that starts at $z_1$ and ends at $z_2$, such that all mid-edges in $\Omega_6^{\delta}$ visited by $\mathscr{F}(\gamma)$ are $z^{(0)},\cdots,z^{(\ell)}$ consecutively, see in Figure~\ref{fig::corr_SAW}.
If we assume further that $\gamma$ is a self-avoiding path on $\Omega_{\312}^{\delta}$, we claim that $\mathscr{F}(\gamma)$ is also a self-avoiding path on $\Omega_6^{\delta}$. 
In this case, the mid-edges $z^{(0)},\cdots,z^{(\ell)}$ are distinct.
For each $0\le j\le \ell-1$, let $v^{(j)}$ be the vertex of $\Omega_6^\delta$ lies between $z^{(j)}$ and $z^{(j+1)}$ on $\mathscr{F}(\gamma)$. It suffices to prove that $v^{(0)},\cdots,v^{(\ell-1)}$ are distinct. 
We prove this by contradiction. Let $\gamma^{(j)}$ be the restriction of $\gamma$ from $z^{(j)}$ to $z^{(j+1)}$, then $\gamma^{(j)}$ visits at least two points in $\mathscr{F}^{-1}(v^{(j)})$.
Suppose there is a vertex $v\in V(\Omega_6^\delta)$ such that $v^{(j_1)}= v^{(j_2)}=v$ for some $j_1\neq j_2$. Since 
$\gamma^{(j_1)}$ and $\gamma^{(j_2)}$ are disjoint, there must be  $\# \mathscr{F}^{-1}(v)\ge 4$, which contradicts with $\# \mathscr{F}^{-1}(v)=3$. 

\begin{figure}[!ht]
    \includegraphics[width=0.4\textwidth]{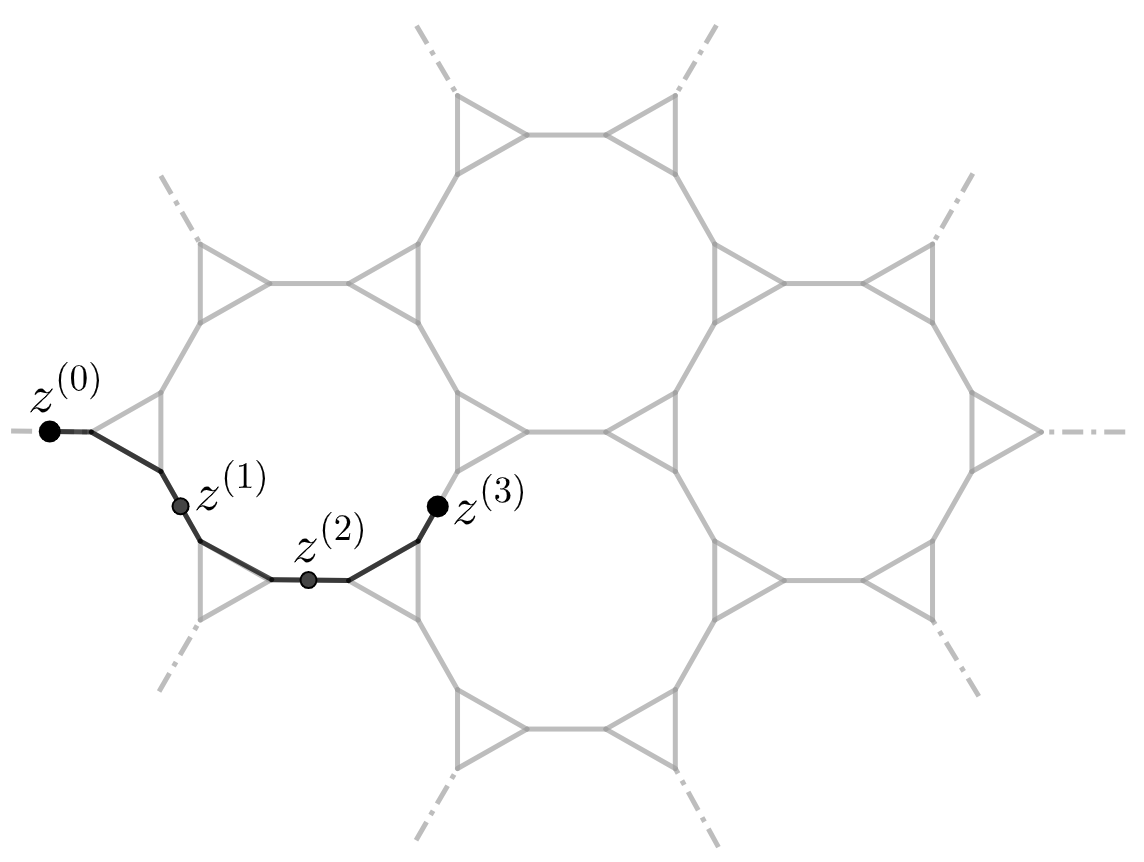}
    \hspace{0.05\textwidth}
    \includegraphics[width=0.4\textwidth]{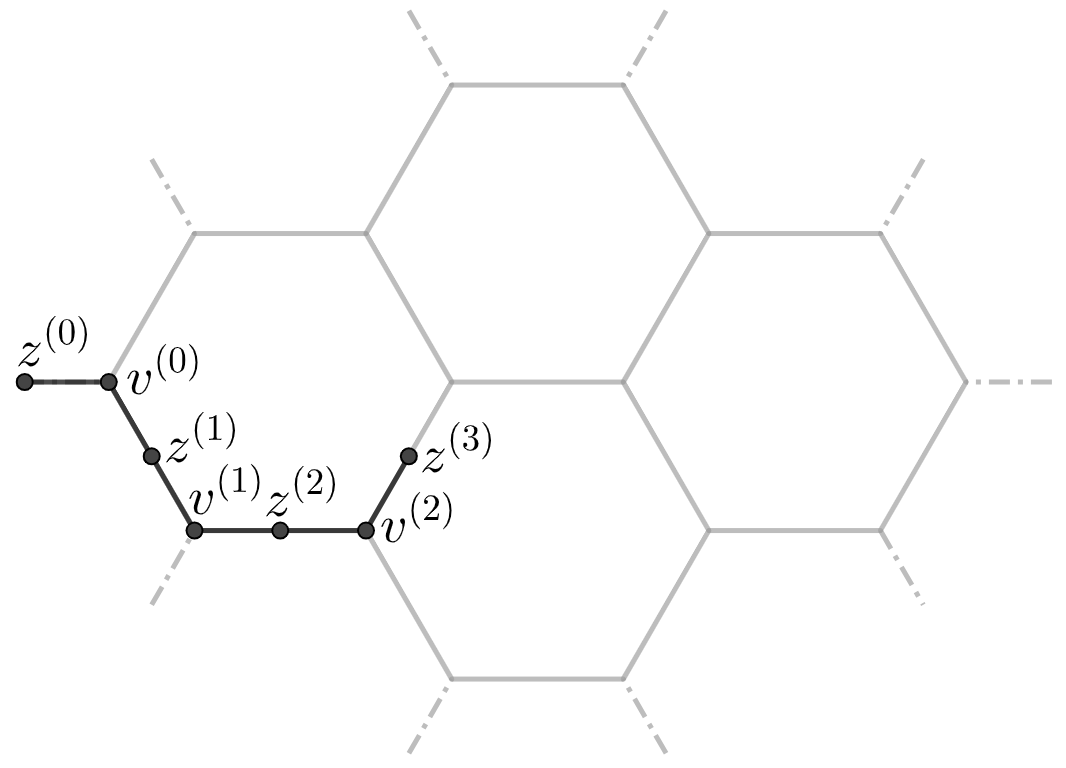}
    \caption{In the left panel, this is a self-avoiding path $\gamma$ on $\Omega_{\312}^\delta$ starting at $z_1$ and ending at $z_2$. In the right panel, this is the corresponding self-avoiding path $\mathscr{F}(\gamma)$ on $\Omega_{6}^\delta$. Note the positions of $z^{(0)},\cdots,z^{(3)}$ and $v^{(0)},v^{(1)},v^{(2)}$.}
    \label{fig::corr_SAW}
\end{figure}

\begin{figure}[!ht]
    \includegraphics[width=0.4\textwidth]{figures/interface_312.png}
    \hspace{0.05\textwidth}
    \includegraphics[width=0.4\textwidth]{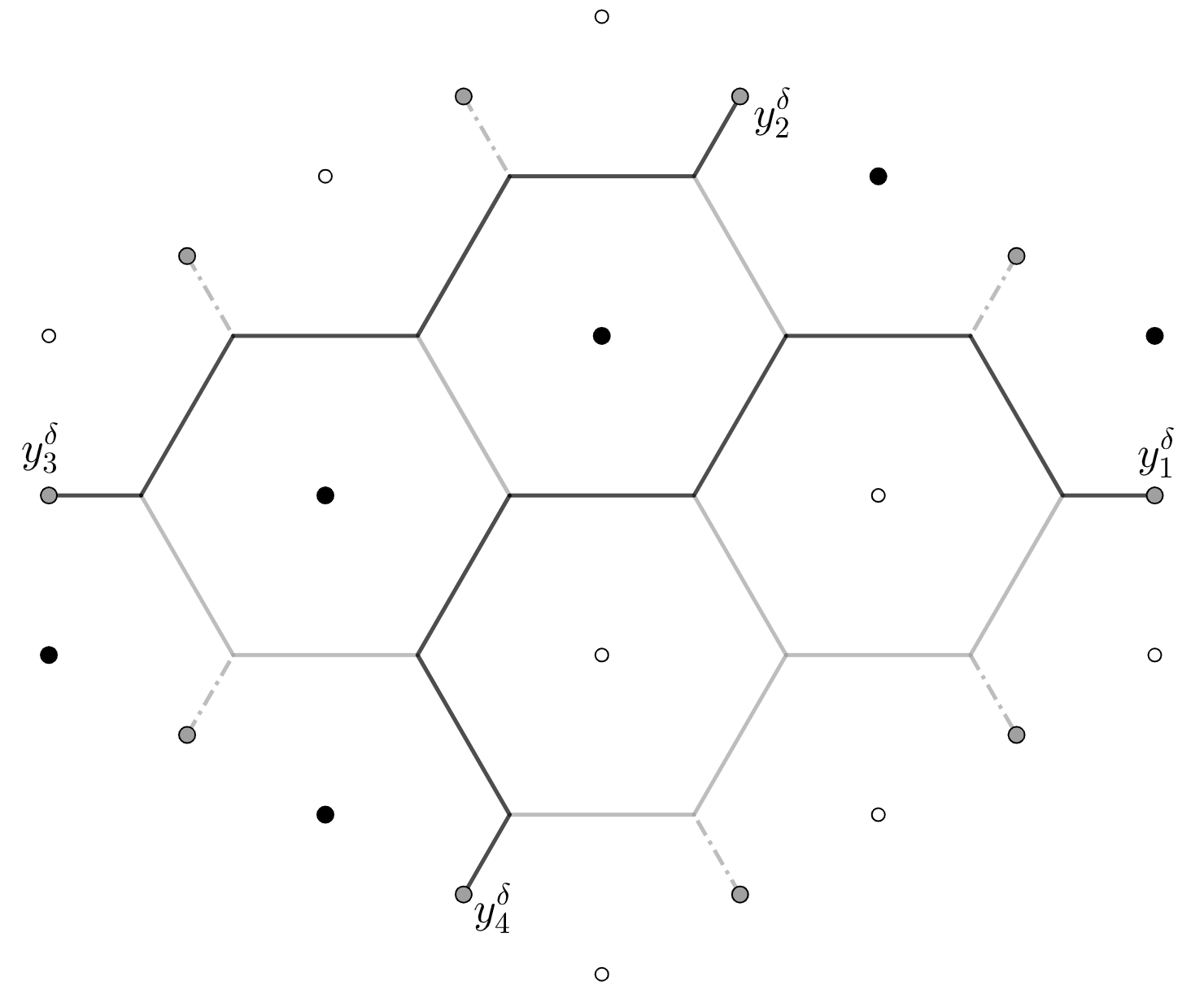}
    \caption{
In the left panel, this is the multiple interfaces for a spin assignment $\sigma_{\312}$ on $\Omega_{\312}^{\delta,*}$ in Figure~\ref{fig::discrete_312}. In the right panel, this is the multiple interfaces for the spin assignment $\sigma_6=\mathscr{G}(\sigma_{\312})$ on $\Omega_6^{\delta,*}$. Vertices with spin $+1$ are colored black and vertices with spin $-1$ are colored white. The multiple interfaces are formed by black edges. In both graphs, $\LA_6^\delta=\LA_{\312}^\delta=\{\{1,4\},\{2,3\}\}$.} 
    \label{fig::interfaces}
\end{figure}

\begin{lemma}\label{lem::corr_Isinginterface}
Assume the same setup as in Lemma~\ref{lem::couplingIsing}. 
Let $(\gamma_{\312}^{\delta,1},\ldots,\gamma_{\312}^{\delta,N})$ be the multiple interfaces of the spin assignment $\sigma_{\312}$ on $\Omega_{\312}^{\delta,*}$ and denote their link pattern by $\LA_{\312}^{\delta}$. 
Denote by $\gamma_6^{\delta,j}=\mathscr{F}(\gamma_{\312}^{\delta,j})$ for $j=1,\ldots,N$ and denote their link pattern by $\LA_6^{\delta}$, 
then $(\gamma_{6}^{\delta,1},\ldots,\gamma_{6}^{\delta,N})$ is the multiple interfaces of the spin assignment $\sigma_6=\mathscr{G}(\sigma_{\312})$ on $\Omega_{6}^{\delta,*}$ and $\LA_6^{\delta}=\LA_{\312}^{\delta}$.    
\end{lemma}

\begin{proof}
    We claim that for each oriented edge $e$ in $\gamma_6^{\delta,j}$, the center of the face on the right-hand side of $e$ has spin $+1$ and the center of the face on the left-hand side has spin $-1$. By definition, there exists $i$ such that $e=(\mathscr{F}(\gamma_{\312,i}^{\delta,j}),\mathscr{F}(\gamma_{\312,i+1}^{\delta,j}))$. Then the edge $e_{\312}=(\gamma_{\312,i}^{\delta,j},\gamma_{\312,i+1}^{\delta,j})$ separates two dodecagons and the center of the dodecagon on the right-hand side has spin $+1$ while the center of the dodecagon on the left-hand side has spin $-1$. These centers are exactly the centers of the faces in $\Omega_6^{\delta}$ on the left-hand side and the right-hand side of $e$. Thus the claim is proved.

    Since for any $1\le j\le N$, the path $\gamma_6^{\delta,j}$ starts from $y_{2j-1}$ and ends at $y_{2k_j}$, we deduce that the link pattern for $(\gamma_{6}^{\delta,1},\ldots,\gamma_{6}^{\delta,N})$ is $\LA_6^{\delta}=\LA_{\312}^{\delta}$. From the claim, $(\gamma_{6}^{\delta,1},\ldots,\gamma_{6}^{\delta,N})$ is the multiple interfaces of $\sigma_6$.
\end{proof}

\begin{proof}[Proof of Proposition~\ref{prop::Ising_crossproba}]
Set
\[\beta_6=\frac{1}{4}\log 3, \qquad \beta_{\312}=\frac{1}{2}\log \frac{1+\sqrt{4\sqrt{3}-3}}{2}.\]
Note that $\beta_6$ is the critical inverse temperature for Ising model on the hexagonal lattice~\cite{HOUTAPPEL1950425}.
Note also that $\beta_{\312}$ is related to $\beta_6$ as in Lemma~\ref{lem::couplingIsing}. 
Combining with Lemma~\ref{lem::corr_Isinginterface}, $\LA_{\312}^\delta=\LA_6^\delta$ has the law of the link pattern for critical Ising model on $\Omega_6^{\delta,*}$ with alternating boundary conditions.  
From the results in~\cite{ChelkakSmirnovIsing} and~\cite{PeltolaWuCrossingProbaIsing}, we have~\eqref{eqn::Ising_crossproba} for $\LA_6^\delta$.
It gives the conclusion for $\LA_{\312}^\delta$ as desired. 
\end{proof}

\subsection{Convergence of observable: proof of Theorem~\ref{thm::Ising_observable}}

Fix a $2$-polygon $(\Omega; A, B)$ and assume that $\Omega$ is flat at $B$. 
Suppose $(\Omega_{\312}^\delta;A^\delta,B^\delta)$ is a family of discrete $2$-polygons such that $(\Omega_{\312}^\delta;A^\delta,B^\delta)$ converges to $(\Omega; A, B)$ in the Carath\'{e}odory sense as $\delta\to 0$ and also assume that $\Omega_{\312}^\delta$ is flat near $B^\delta$. For inverse temperature $\beta$, we set $x=\ee^{-2\beta}$. 
Suppose $z^\diamond$ is a midpoint of an edge in $\Omega_{\312}^\delta$ separating two dodecagons. Recall that $\LG_{\312}^\delta(z_1,z_2)$ is defined in Section~\ref{subsec::Intro_Ising}, and the observable $\IO_{\312}^{\delta}(x; z^{\diamond})$ is defined in~\eqref{eqn::Ising_observable}. 
We define analogous notions for the hexagonal lattice. 
For given two mid-edges $z_1,z_2$ in $\Omega_6^\delta$, we define $\LG_6^\delta(z_1,z_2)$ to be the collection of subgraphs of $\Omega_{6}^\delta\cup \{z_1,z_2\}$ that consists of some loops and a self-avoiding path from $z_1$ to $z_2$, such that the loops and the self-avoiding path are disjoint. The fermionic observable for Ising model with inverse temperature $\beta$ on $\Omega_6^\delta$ is defined as 
\begin{align*}
    \IO_6^\delta(x;z^\diamond)=\dfrac{\sum\limits_{w\in \LG_{6}^\delta(A^\delta,z^\diamond)}x^{|w|}\exp\left(-\frac{\ii}{2}W_{\gamma(w)}(A^\delta,z^\diamond)\right)}{\sum\limits_{w\in \LG_{6}^\delta(A^\delta,B^\delta)}x^{|w|}\exp\left(-\frac{\ii}{2}W_{\gamma(w)}(A^\delta,B^\delta)\right)}, \qquad \text{where }x=\ee^{-2\beta}, 
\end{align*}
and $z^\diamond$ is a midpoint of an edge in $\Omega_{6}^\delta$ and $|w|$ denotes the number of vertices of $\Omega_6^\delta$ in $w$.
\begin{lemma}\label{lem::corr_Ising_observable}
    Let $z^{\diamond}$ be a midpoint of an edge in $\Omega_{\312}^\delta$ separating two dodecagons, we have 
\begin{align}\label{eqn::corr_Ising_observable}
    \IO_{\312}^\delta(x;z^\diamond)=\IO_6^\delta\left(\frac{x^3+x^2}{x^3+1};z^\diamond\right).
\end{align}    
\end{lemma}
\begin{proof}
Suppose $z_1,z_2$ are two midpoints of edges in $\Omega_{\312}^\delta$ separating two dodecagons.
Recall that $\mathscr{F}$ is the mapping on paths defined at the beginning of Section~\ref{subsec::Ising_crossproba}, we extend it as a mapping from $\LG_{\312}^\delta(z_1,z_2)$ to $ \LG_{6}^\delta(z_1,z_2)$ as follows.
Suppose $w\in \LG_{\312}^\delta(z_1,z_2)$ contains loops $w_1,\cdots,w_n$ and a self-avoiding path $w_0=\gamma(w)$ from $z_1$ to $z_2$. 
We define $\mathscr{F}(w)$ to be the union of $\mathscr{F}(w_0),\cdots,\mathscr{F}(w_n)$, see in Figure~\ref{fig::corr_loop_path}. Here if $w_i$ is a loop around a triangle in $\Omega_{\312}^\delta$, we define $\mathscr{F}(w_i)$ to be an empty graph.
\begin{figure}[!ht]
    \includegraphics[width=0.42\textwidth]{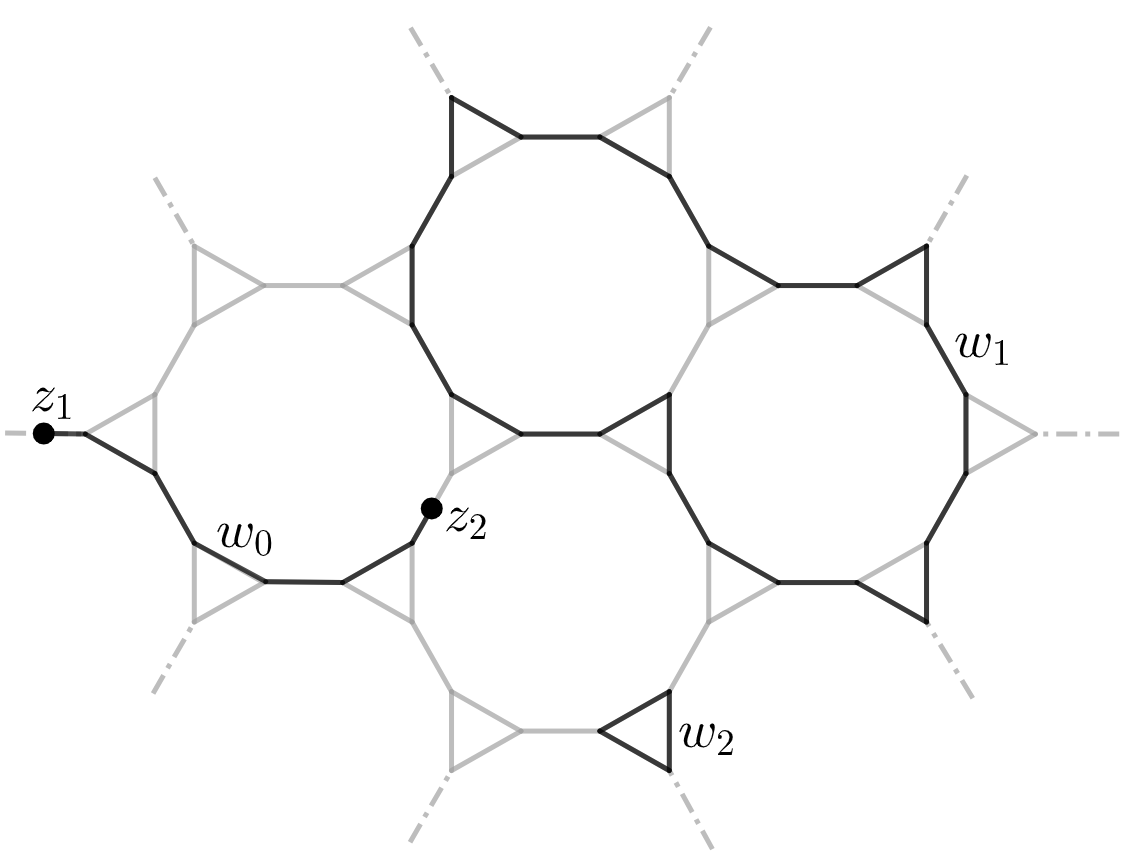}
    \includegraphics[width=0.42\textwidth]{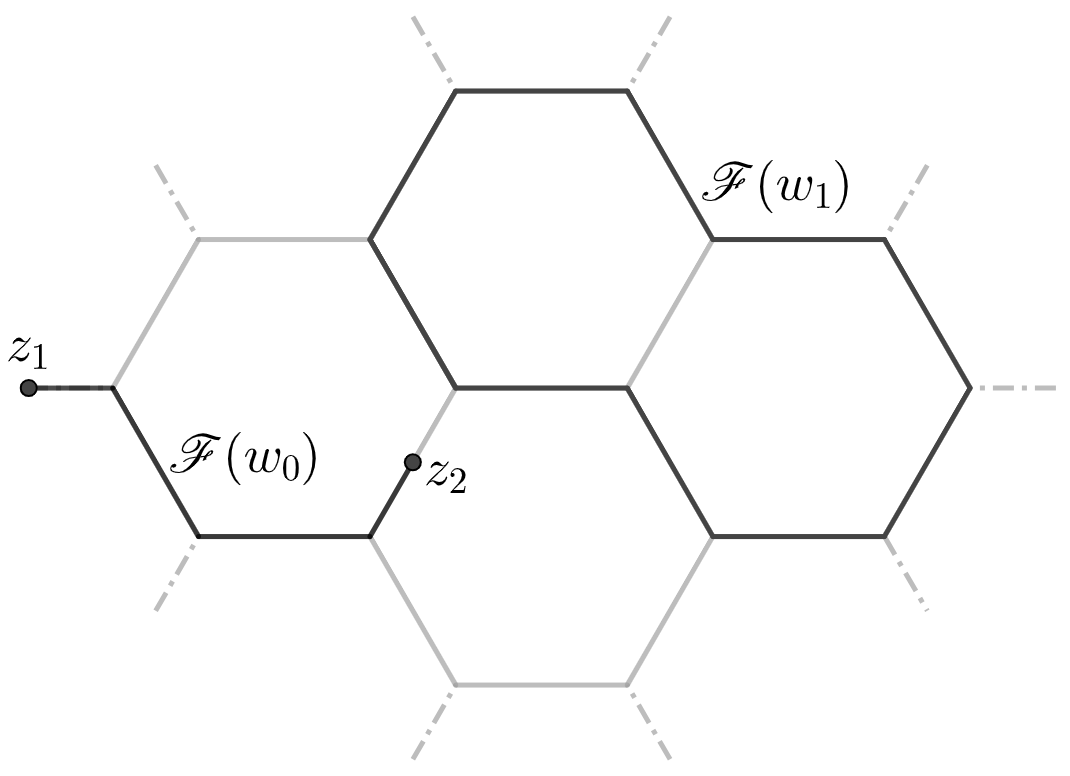}
    \caption{In the left panel, this is a configuration $w\in \LG_{\312}^\delta(z_1,z_2)$ contains two loops $w_1,w_2$ and a self-avoiding path $w_0$ from $z_1$ to $z_2$, where $w_1$ surrounds two dodecagons and $w_2$ surrounds a triangle. In the right panel, this is the corresponding $\mathscr{F}(w)\in \LG_6^\delta(z_1,z_2)$. Note that $\mathscr{F}(w_0)$ is a self-avoiding path from $z_1$ to $z_2$, and $\mathscr{F}(w_1)$ is a loop surrounds two hexagons, and $\mathscr{F}(w_2)$ is an empty graph.}
    \label{fig::corr_loop_path}
\end{figure}

Let us check the image of the mapping $\mathscr{F}$ and its inverse.
\begin{itemize}
\item On the one hand, for $w\in \LG_{\312}^\delta(z_1,z_2)$, we claim that $\mathscr{F}(w)$ belongs to $\LG_{6}^\delta(z_1,z_2)$. Recall that $\mathscr{F}(w_0)$ is self-avoiding. Using a similar argument as in the proof that $\mathscr{F}(w_0)$ is self-avoiding, one may check that $\mathscr{F}(w_i)$ is a simple loop for $i\ge 1$ and $\mathscr{F}(w_0),\cdots,\mathscr{F}(w_n)$ are disjoint. Therefore $\mathscr{F}(w)\in\LG_{6}^\delta(z_1,z_2)$.  
\item On the other hand, for $\tilde{w}\in \LG_{6}^\delta(z_1,z_2)$, let us consider subgraphs $w\in \LG_{\312}^\delta(z_1,z_2)$ such that $\mathscr{F}(w)=\tilde{w}$. For a mid-edge $z$ in $\Omega_{\312}^\delta$ separating two dodecagons, $z$ belongs to $w$ if and only if $z$ belongs $\tilde{w}$. For a vertex $v$ of $\Omega_6^\delta$, the inverse $\mathscr{F}^{-1}(v)$ contains three vertices of $\Omega_{\312}^\delta$ that form a triangle $\bigtriangleup_v$. Let $w_v=w\cap \bigtriangleup_v$, then $w$ is uniquely determined by $(w_v:v\in \Omega_6^{\delta})$. Denote by $|w_v|$ the number of edges in $w_v$, there are two cases. 
\begin{itemize}
\item  If $v$ does not belong to $w$, none of the edges adjacent to $\bigtriangleup_v$ belongs to $w$. There are two possibilities for $w_v$ as shown in Figure~\ref{fig::corr_Isingconfig}~(a). For the first case, $|w_v|=0$ ; for the second case, $|w_v|=3$.
\item If $v$ belongs to $w$, two of the three edges adjacent to $\bigtriangleup_v$ belong to $w$. There are two possibilities for $w_v$ as shown in Figure~\ref{fig::corr_Isingconfig}~(b). For the first case, $|w_v|=1$; for the second case, $|w_v|=2$. 
\end{itemize}
Notice that
\[|w|=|\tilde{w}|+\sum_{v\in V(\Omega_6^\delta)}|w_v|, \qquad W_{\gamma(\tilde{w})}(z_1,z_2)=W_{\gamma(w)}(z_1,z_2).\]
\end{itemize}

From the above analysis, for $\tilde{w}\in \LG_{6}^\delta(z_1,z_2)$, we have
\begin{align*}
\sum_{\substack{w\in \LG_{\312}^\delta(z_1,z_2),\\ \mathscr{F}(w)=\tilde{w}}}\exp\left(-\frac{\ii}{2}W_{\gamma(w)}(z_1,z_2)\right)x^{|w|}
&=\exp\left(-\frac{\ii}{2}W_{\gamma(\tilde{w})}(z_1,z_2)\right)\sum_{\substack{w\in \LG_{\312}^\delta(z_1,z_2),\\ \mathscr{F}(w)=\tilde{w}}}x^{|w|}\\
&=\exp\left(-\frac{\ii}{2}W_{\gamma(\tilde{w})}(z_1,z_2)\right)x^{|\tilde{w}|}\prod_{v\in V(\Omega_6^\delta)}\left(\sum_{w_v}x^{|w_v|}\right)
\\ &=\exp\left(-\frac{\ii}{2}W_{\gamma(\tilde{w})}(z_1,z_2)\right)(x^3+x^2)^{|\tilde{w}|}(x^3+1)^{m-|\tilde{w}|},
\end{align*}
where $m$ is the number of vertices in $\Omega_6^\delta$. Summing over all configurations $\tilde{w}$, we find 
\[\sum_{w\in \LG_{\312}^\delta(z_1,z_2)}\exp\left(-\frac{\ii}{2}W_{\gamma(w)}(z_1,z_2)\right)x^{|w|}=(x^3+1)^m \sum_{\tilde{w}\in \LG_6^\delta(z_1,z_2)}\exp\left(-\frac{\ii}{2}W_{\gamma(\tilde{w})}(z_1,z_2)\right)\left(\frac{x^3+x^2}{x^3+1}\right)^{|\tilde{w}|}.\]
Plugging into~\eqref{eqn::Ising_observable}, we obtain~\eqref{eqn::corr_Ising_observable} as desired.
\end{proof}

\begin{figure}[ht!]
\begin{subfigure}[t]{0.48\textwidth}
\begin{center}
\includegraphics[height=0.2\textwidth]{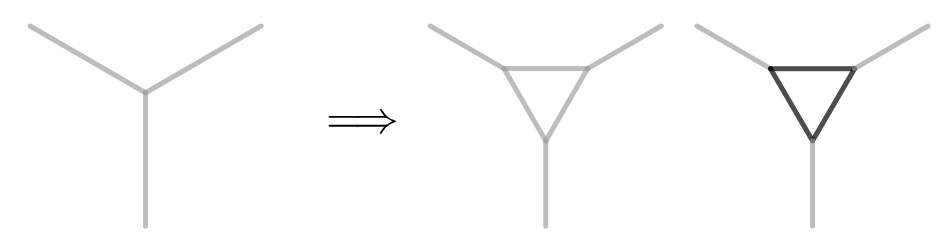}
\end{center}
\caption{Two choices for $w_v$ when $v$ does not belong to $w$.}
\end{subfigure}
\begin{subfigure}[t]{0.48\textwidth}
\begin{center}
\includegraphics[height=0.2\textwidth]{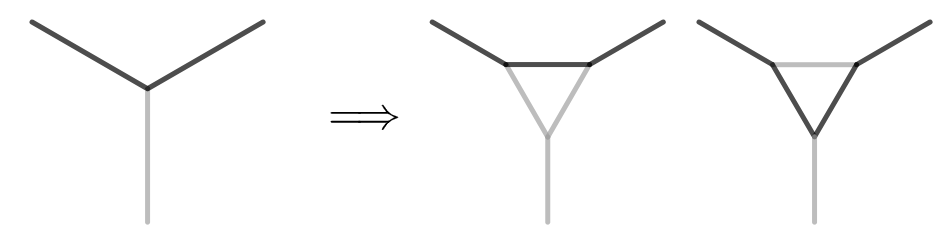}
\end{center}
\caption{Two choices for $w_v$ when $v$ belongs to $w$.}
\end{subfigure}
\caption{Possible configurations for $w_v$.}
\label{fig::corr_Isingconfig}
\end{figure}

\begin{proof}[Proof of Theorem~\ref{thm::Ising_observable}]
Note that $\beta_6=\frac{1}{4}\log 3$ is the critical inverse temperature for Ising model on the hexagonal lattice~\cite{HOUTAPPEL1950425}.
Set 
\[x_6=\frac{\sqrt{3}}{3}, \qquad x_{\312}=\frac{2}{1+\sqrt{4\sqrt{3}-3}}. \]
Then $x_6=\ee^{-2\beta_6}$ and
\[x_6=\frac{x_{\312}^3+x_{\312}^2}{x_{\312}^3+1}.\]
From Lemma~\ref{lem::corr_Ising_observable}, we have $\IO_{\312}^\delta(x_{\312};z^\diamond)=\IO_6^\delta\left(x_6;z^\diamond\right)$. 
The convergence of $\IO_{\312}^\delta(x_{\312};z^\diamond)$ follows from the convergence of $\IO_6^\delta\left(x_6;z^\diamond\right)$ proved in~\cite{ChelkakSmirnovIsing}.
\end{proof}

%% file: tex/perco.tex
Consider the site percolation with $p=1/2$ on $\Omega_{\312}^{\delta, *}$. 
For a vertex $z\in V(\Omega_{\312}^{\delta})$, recall that the event $\LE_{\312}^{\delta}(A; z)$ is defined in Section~\ref{subsec::intro_perco} and $\PO_{\312}^{\delta}(A; z)$ is the probability of the event $\LE_{\312}^{\delta}(A; z)$. If a path $\gamma$ in $\Omega_{\312}^{\delta,*}$ satisfies the conditions in the definition of $\LE_{\312}^\delta(A;z)$, we say the event $\LE_{\312}^\delta(A;z)$ is held by $\gamma$. We define the events $\LE_{\312}^\delta(B;z),\LE_{\312}^\delta(C;z)$ and the probabilities $\PO_{\312}^\delta(B;z),\PO_{\312}^\delta(C;z)$ in a similar way. 
Recall that the observable $\PO_{\312}^{\delta}$ is defined in~\eqref{eqn::perco_observable}. 
We consider the restriction of the percolation to $\Omega_6^{\delta, *}$. For $z\in V(\Omega_6^\delta)$, we define $\LE_{6}^{\delta}(A; z)$ to be the event that there exists a path $\gamma$ on $\Omega_{6}^{\delta,*}$ connecting boundary arcs $(A^\delta B^\delta)$ and $(C^\delta A^\delta)$, such that all vertices on $\gamma$ have spin $+1$ and $\gamma$ separates $z$ from the boundary arc $(B^\delta C^\delta)$. We denote by $\PO_6^{\delta}(A; z)$
the probability of $\LE_{6}^{\delta}(A; z)$. We define the events $\LE_{6}^\delta(B;z),\LE_{6}^\delta(C;z)$ and the probabilities $\PO_{6}^\delta(B;z),\PO_{6}^\delta(C;z)$ in a similar way.
We also define the observable for the hexagonal lattice in a similar way:
\begin{align}\label{eqn::perco_observable_6}
\PO_{6}^{\delta}(z):=\PO_{6}^{\delta}(A; z)+\tau \PO_{6}^{\delta}(B; z)+\tau^2 \PO_{6}^{\delta}(C; z), \qquad \text{where }\tau=\ee^{2\pi\ii/3}. 
\end{align}
\begin{lemma}\label{lem::perco_observable_close}
There exist constants $\eps>0$ and $C\in (0, \infty)$ such that 
\begin{align}\label{eqn::perco_observable_close}
    |\PO_{\312}^\delta(z)-\PO_{6}^\delta(\mathscr{F}(z))|\le C\delta^{\eps},\quad \forall \delta>0,z\in V(\Omega_{\312}^{\delta,*}).
\end{align}
\end{lemma}
\begin{proof}
    It suffices to prove that there exist constants $\eps>0$ and $C\in (0,\infty)$ such that
    \begin{align}\label{eqn::perco_observable_close_auxA}
    |\PO_{\312}^\delta(A;z)-\PO_{6}^\delta(A;\mathscr{F}(z))|\le C\delta^{\eps},\quad \forall \delta>0,z\in V(\Omega_{\312}^{\delta,*}).
    \end{align}
Let $\Sigma_{\312}$ be the site percolation with $p=1/2$ on $\Omega_{\312}^{\delta,*}$. Then the restriction $\Sigma_6=\mathscr{G}(\Sigma_{\312})=(\Sigma_{\312}^v:v\in V(\Omega_6^{\delta,*}))$ is the site percolation with $p=1/2$ on $\Omega_6^{\delta,*}$.
Fix $z\in V(\Omega_{\312}^{\delta, *})$ and let $\tilde{z}=\mathscr{F}(z)$ be the corresponding vertex of $z$ on $\Omega_6^{\delta}$. Let us compare the two events $\LE_{\312}^{\delta}(A;z)$ and $\LE_{6}^{\delta}(A;\tilde{z})$.
\begin{itemize}
\item On the one hand, suppose $\Sigma_6$ satisfies $\LE_{6}^{\delta}(A;\tilde{z})$. Suppose $\LE_{6}^\delta(A;\tilde{z})$ is held by the curve $\gamma$, then $\gamma$ is also a path on $\Omega_{\312}^{\delta,*}$ with spins $+1$ along the path. Furthermore, since $\tilde{z}$ and $z$ lie in the same face of $\Omega_{6}^{\delta,*}$, the path $\gamma$ separates $z$ from $(B^\delta C^\delta)$. Therefore, $\Sigma_{\312}$ satisfies $\LE_{\312}^{\delta}(A;z)$.
\item On the other hand, suppose $\Sigma_{\312}$ satisfies $\LE_{\312}^\delta(A;z)$ while $\Sigma_6$ does not satisfy $\LE_{6}^\delta(A;\tilde{z})$. Let $P,Q,R$ be the three vertices of the equilateral triangle in $\Omega_6^{\delta,*}$ with $\tilde{z}$ as the center. 
As $z$ is a vertex of $\Omega_{\312}^{\delta}$, it lies in one of the three faces around $\tilde{z}$,
and without loss of generality, suppose $z$ lies in the triangle $QR\tilde{z}$. Denote by $z'$ the reflection of $\tilde{z}$ with respect to $QR$, we claim that $\LE_{6}^{\delta}(A;z')$ holds. Suppose $\LE_{\312}^\delta(A;z)$ is held by $\gamma$, and let $\tilde{\gamma}$ be the path on $\Omega_6^{\delta,\ast}$ that visits all vertices of $\Omega_6^{\delta,\ast}$ on $\gamma$ sequentially. By definition, $\tilde{\gamma}$ has spins $+1$ along the path connecting $(A^\delta B^\delta)$ and $(C^\delta A^\delta)$.  By our assumptions, $A^\delta$ and $z$ lie on the same side of $\gamma$ while $A^\delta$ and $\tilde{z}$ lie on opposite sides of $\tilde{\gamma}$. Since $\gamma$ is generated from $\tilde{\gamma}$ by adding some centers of the faces in $\Omega_6^{\delta,*}$, the only possible situation is that $\tilde{\gamma}$ visits the edges $[Q\tilde{z}]$ and $[\tilde{z}R]$, see in Figure~\ref{fig::perco_interface}. In this case, $[QR]$ is on $\tilde{\gamma}$, and $\tilde{z},z'$ lie on opposite sides of $\tilde{\gamma}$. As a consequence, $\LE_{6}^{\delta}(A;z')$ is held by $\tilde{\gamma}$.
\end{itemize}
The above analysis gives
    \begin{align*}
        0\le \PO_{\312}^\delta(A;z)-\PO_{6}^\delta(A;\tilde{z})\le \mathbb{P}[\LE_{6}^\delta(A;z')\backslash \LE_{6}^\delta(A;\tilde{z})].
    \end{align*}
    From~\cite{SmirnovPercolationConformalInvariance}, there exist constants $C>0$ and $\eps>0$ such that
    \[\mathbb{P}[\LE_{6}^\delta(A;z')\backslash \LE_{6}^\delta(A;\tilde{z})]\le C|z'-\tilde{z}|^\eps,\]
    as a consequence of the Russo-Seymour-Welsh theory. Since $|z'-\tilde{z}|$ is propotional to $\delta$, these complete the proof for~\eqref{eqn::perco_observable_close_auxA}, and hence complete the proof for~\eqref{eqn::perco_observable_close}. 
\end{proof}
  \begin{figure}[!h]
        \includegraphics[width=0.8\textwidth]{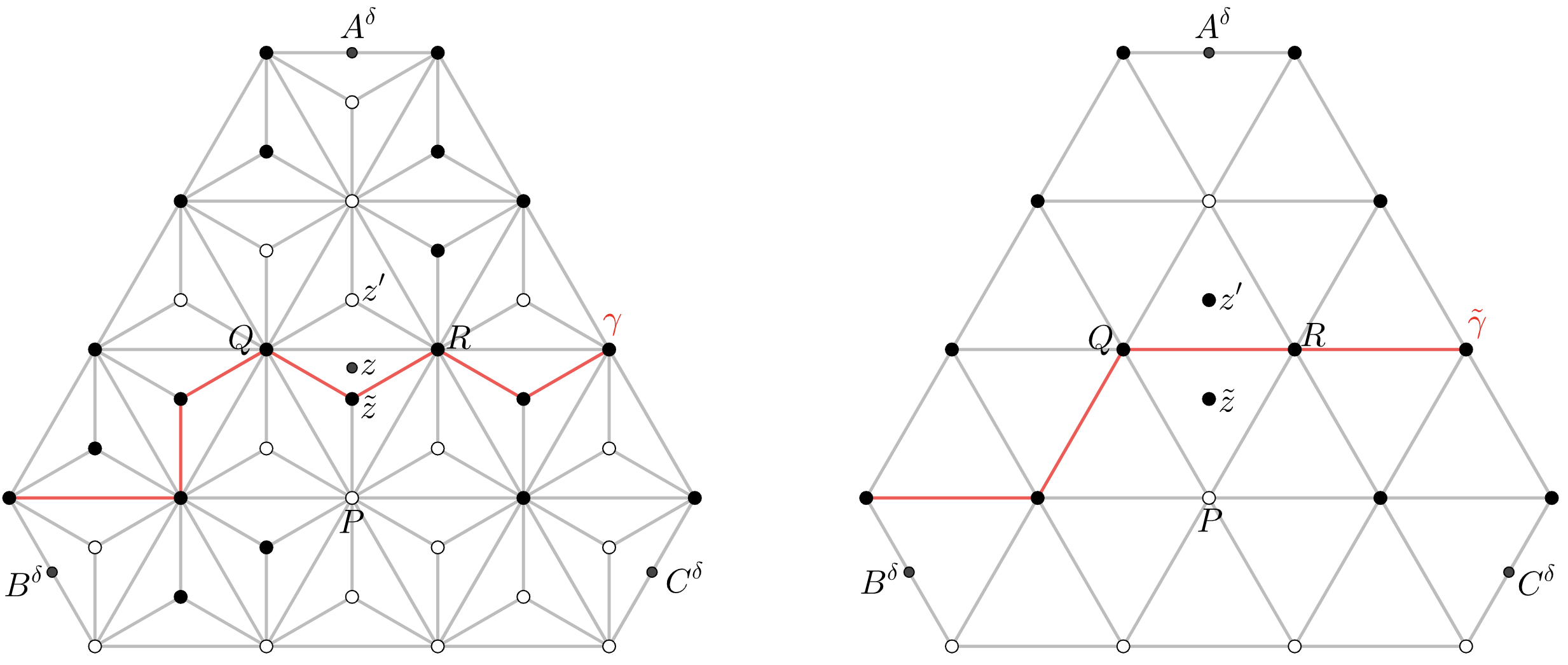}
        \caption{Crossing paths of site percolation. The left panel is the site percolation on $\Omega_{\312}^{\delta,*}$ with a crossing path $\gamma$, and the right panel is the corresponding site percolation on $\Omega_6^{\delta,*}$ with the corresponding crossing path $\tilde{\gamma}$.}
        \label{fig::perco_interface}
    \end{figure}
\begin{proof}[Proof of Theorem~\ref{thm::Perco_observable}]
    It was shown in~\cite{SmirnovPercolationConformalInvariance, beffara2007cardy} that $\PO_6^\delta$ converges to $\PO$ uniformly as $\delta\to 0$. Combining with Lemma~\ref{lem::perco_observable_close}, $\PO_{\312}^\delta $ also converges to $\PO$ uniformly. This completes the proof.
\end{proof}

\begin{proof}[Proof of Proposition~\ref{prop::Perco_crossproba}]

Let $\Sigma_{\312}$ be the site percolation with $p=1/2$ on $\Omega_{\312}^{\delta,*}$ with alternating boundary conditions. Then the restriction $\Sigma_6=\mathscr{G}(\Sigma_{\312})=(\Sigma_{\312}^v:v\in V(\Omega_6^{\delta,*}))$ is the site percolation with $p=1/2$ on $\Omega_6^{\delta,*}$ with alternating boundary conditions. Combining with Lemma~\ref{lem::corr_Isinginterface}, $\LA_{\312}^\delta=\LA_6^\delta$ has the law of the link pattern for critical site percolation on $\Omega_6^{\delta,*}$ with alternating boundary conditions. From the conclusion for percolation on $\Omega_6^{\delta, *}$ (although it is not explicitly written out, the conclusion could be proved using a similar (and simpler) proof as for FK-Ising model~\cite[Proposition~4.2]{FengPeltolaWuConnectionProbaFKIsing}), we obtain~\eqref{eqn::perco_crossproba} for $\Omega_{\312}^{\delta, *}$. 
\end{proof}

%% file: tex/otherlattice.tex
\paragraph*{Appollonian gasket.}
We denote by $G_1=\L_6$ the hexagonal lattice and by $G_2=\L_{\312}$ the Fisher transformation of $G_1$. 
Generally, we denote by $G_{n+1}$ the Fisher transformation of $G_n$ for $n\ge 1$. The limiting graph of $G_n$ is the Apollonian gasket (when the positions of new vertices are taken properly), see Figure~\ref{fig::Appollonian_Martini}~(a). 


Our proof in Sections~\ref{sec::Ising}-\ref{sec::perco} also gives the conformal invariance of the Ising model and of the percolation on $G_n$ for $n\ge 2$. 
In the case of the percolation, the critical parameter on $G_n$ is always $1/2$. 
In the case of the Ising model, the critical parameter $\alpha_n$ for $G_n$ can be defined recursively: 
\begin{equation}\label{eqn::Ising_critical_recursion}
\alpha_1=\frac{1}{4}\log 3, \qquad \alpha_n=\frac{1}{2}\log\left(\ee^{4\alpha_{n+1}}-\ee^{2\alpha_{n+1}}+1\right), \qquad n\ge 1. 
\end{equation}
The parameter $x_n$ for Chelkak-Smirnov's observable on $G_n$ can be defined recursively as well: 
\begin{equation}\label{eqn::Ising_x_recursion}
x_1=\frac{\sqrt{3}}{3}, \qquad x_n=\frac{x^3_{n+1}+x^2_{n+1}}{x^3_{n+1}+1}, \qquad n\ge 1. 
\end{equation}
These parameters have limits as $n\to\infty$: 
\begin{align}\label{eqn::alphan_limit}
\alpha_n=\frac{1}{2n}-\frac{\log n}{2n^2}+O\left(\frac{1}{n^2}\right), \qquad x_n=1-\frac{1}{n}+O\left(\frac{\log n}{n^2}\right). 
\end{align}
\begin{figure}[ht!]
\begin{subfigure}[t]{0.48\textwidth}
\begin{center}
  \includegraphics[height=0.2\textheight]{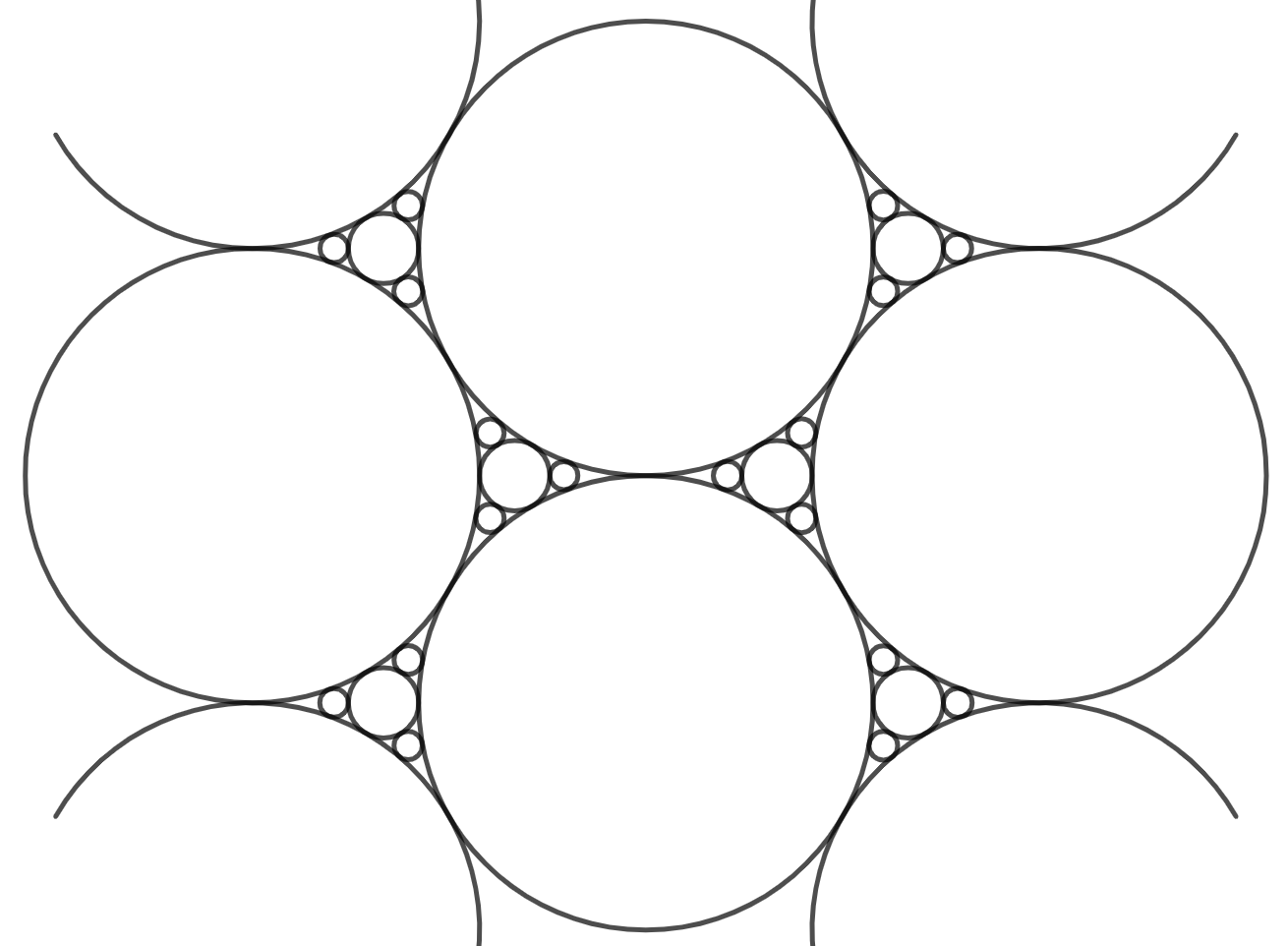}
\end{center}
\caption{The Apollonian gasket.}
\end{subfigure}
\begin{subfigure}[t]{0.48\textwidth}
\begin{center}
\includegraphics[height=0.2\textheight]{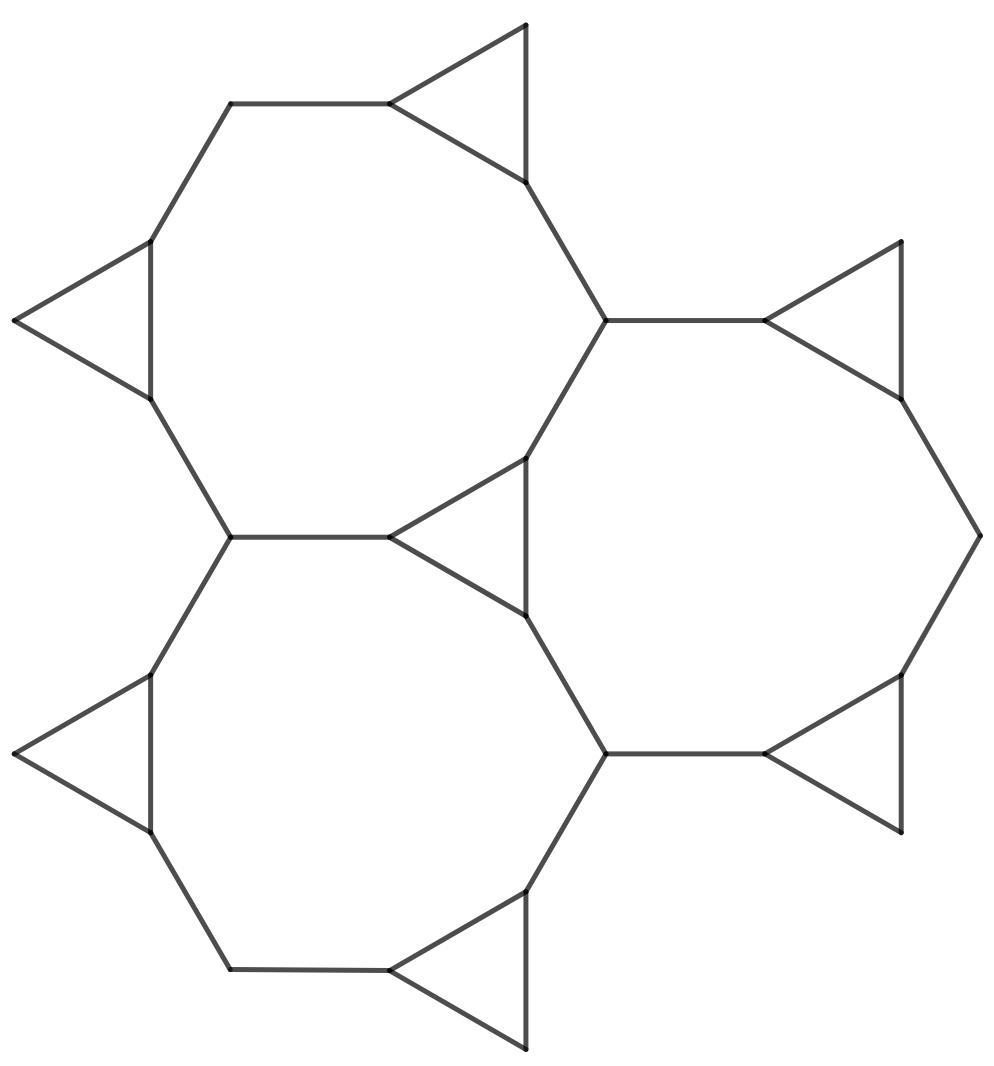}
\end{center}
\caption{The martini lattice.}
\end{subfigure}
\caption{\label{fig::Appollonian_Martini}}
\end{figure}

\paragraph*{Martini lattice.}
The martini lattice is the Fisher transformation of the hexagonal lattice at alternate vertices, see Figure~\ref{fig::Appollonian_Martini}~(b). 
See~\cite{grimmett2012self, DingGuoOnloopMartini312} for analysis on the connective constant on the martini lattice.
Our proof in Sections~\ref{sec::Ising}-\ref{sec::perco} also gives the conformal invariance of the Ising model and of the percolation on the martini lattice. 
In the case of the percolation, the critical parameter on is $1/2$. 
In the case of the Ising model, the critical parameter $\beta$ for the martini lattice is the unique root of 
\begin{equation}\label{eqn::Ising_critical_martini}
\ee^{6\beta}-\ee^{4\beta}+\ee^{2\beta}=\sqrt{3}.
\end{equation}
The parameter $x$ for Chelkak-Smirnov's observable on the martini lattice is the positive root of
\begin{equation}\label{eqn::Ising_x_martini}
\frac{x^4+x^3}{x^3+1}=\frac{\sqrt{3}}{3}.
\end{equation}


%% file: tex/ppf.tex
Following O. Schramm's seminal idea to describe the scaling limit of interfaces in critical planar lattice models using the Loewner evolution~\cite{SchrammFirstSLE}, one is led to analyzing partition functions for the evolution.
J. Dubédat's commutation relation, which combines conformal invariance with the domain Markov property, implies that natural partition functions are solutions to the Belavin–Polyakov–Zamolodchikov (BPZ) equations~\cite{DubedatCommutationSLE}.
Pure partition functions for multiple SLE form a basis of the solution space of certain BPZ equations.
They have been intensively studied in recent years; see~\cite{FengLiuPeltolaWu2024} and references therein. 
We briefly summarize their definition in this appendix. 

Fix $\kappa\in (0,8)$ and set $h=\frac{6-\kappa}{2\kappa}$. 
Let $\HH=\{z\in \C: \Im{z}>0\}$ be the upper-half plane and denote $\chamber_{2N}=\{(y_1, \ldots, y_{2N}): y_1<\cdots<y_{2N}\}$. The pure partition functions of multiple $\SLE_{\kappa}$ are the recursive collection $\{\LZ_{\alpha}^{(\kappa)}: \alpha\in\LP:=\cup_N\LP_N\}$ of positive functions $\LZ_{\alpha}^{(\kappa)}: \chamber_{2N}\to (0,\infty)$ uniquely determined by the BPZ PDE system~\eqref{eqn::BPZPDE}, M\"obius covariance~\eqref{eqn::COV}, the power-law bound~\eqref{eqn::PLB}, as well as the recursive asymptotics~\eqref{eqn::ASY}. 
\begin{itemize}
    \item BPZ PDE: 
    \begin{align}\label{eqn::BPZPDE}
        \left[\frac{\kappa}{2}\partial_j+\sum_{i\neq j}\left(\frac{2}{y_i-y_j}\partial_i-\frac{2h}{(y_i-y_j)^2}\right)\right]\LZ=0,\quad \text{for all }j\in \{1,\ldots,2N\}.
    \end{align}
    \item M\"obius covariance: 
    \begin{align}\label{eqn::COV}
        \LZ(y_1,\ldots,y_{2N})=\prod_{i=1}^{2N}\varphi'(y_i)^h\times \LZ(\varphi(y_1),\ldots,\varphi(y_{2N})),
    \end{align}
    for all M\"obius transforms $\varphi:\HH \to \HH$ such that $\varphi(y_1)<\cdots<\varphi(y_{2N})$. 
     \item Power-law bound: There exist $C>0$ and $p>0$ such that for all $N\ge 1$ and all $(y_1, \ldots, y_{2N})\in\chamber_{2N}$,  
    \begin{align}\label{eqn::PLB}
    |\LZ(y_1,\ldots,y_{2N})|\le C\prod_{1\le i<j\le 2N}(y_j-y_i)^{\mu_{ij}(p)},\qquad \text{where }\mu_{ij}(p)=\begin{cases}p,&\text{if }|y_j-y_i|>1;\\ -p&\text{if }|y_j-y_i|\le 1.\end{cases}
    \end{align}
    \item Asymptotics: Denote by $\emptyset$ the empty link pattern in $\LP_0$, then $\LZ_\emptyset^{(\kappa)}=1$. For all $N\ge 1$, $\alpha\in \LP_N$, $j\in \{1,\ldots,2N-1\}$ and $\xi\in (y_{j-1},y_{j+2})$ (here with the convention, $y_0=-\infty$ and $y_{2N+1}=+\infty$), 
\begin{align}\label{eqn::ASY}
\lim_{y_j,y_{j+1}\to \xi}\frac{\LZ_\alpha^{(\kappa)}(y_1,\ldots,y_{2N})}{(y_{j+1}-y_j)^{-2h}}=\begin{cases}\LZ_{\hat{\alpha}}^{(\kappa)}(y_1,\ldots,y_{j-1},y_{j+2},\ldots,y_{2N}),&\text{if }\{j,j+1\}\in \alpha;\\ 0,&\text{otherwise};\end{cases}
\end{align}    
    where $\hat{\alpha}\in \LP_{N-1}$ is the link pattern obtained from $\alpha$ by removing the link $\{j,j+1\}\in \alpha$ and relabeling the remaining indices by $\{1,\ldots,2N-2\}$.
\end{itemize}

\begin{definition}\label{def::pure_partition}
For a general polygon $(\Omega;y_1,\ldots,y_{2N})$, the pure partition functions for multiple $\SLE_{\kappa}$ are defined as 
\begin{align*}
    \LZ_{\alpha}^{(\kappa)}(\Omega;y_1,\ldots,y_{2N}):=\prod_{i=1}^{2N}|\varphi'(y_i)|^h\times \LZ_{\alpha}^{(\kappa)}(\varphi(y_1),\ldots,\varphi(y_{2N})), \qquad\text{for }\alpha\in\LP_N, 
\end{align*}
where $\varphi:\Omega\to \HH$ is a conformal mapping such that $\varphi(y_1)<\cdots<\varphi(y_{2N})$. By M\"obius covariance~\eqref{eqn::COV}, this definition is independent of the choice of $\varphi$.
\end{definition}